\documentclass{amsart}

\usepackage[dvips]{graphics}

\usepackage{amsmath,amscd,enumerate,amsthm}
\usepackage{amssymb}
\usepackage{amsfonts}
\usepackage[all,knot]{xy}
\usepackage{color}

\newtheorem{theorem}{Theorem}[section]
\newtheorem{proposition}[theorem]{Proposition}
\newtheorem{corollary}[theorem]{Corollary}
\newtheorem{lemma}[theorem]{Lemma}

\theoremstyle{definition}
\newtheorem{definition}[theorem]{Definition}
\newtheorem{example}[theorem]{Example}

\theoremstyle{remark}
\newtheorem{remark}[theorem]{Remark}

\numberwithin{equation}{section}

\newcommand{\Projection}{\ensuremath{\mathsf{p}}}

\newcommand{\FNo}{\ensuremath{\mathcal{F}_{\newN}^{0}}} 

\newcommand{\FN}{\ensuremath{\mathcal{F}_{\newN}}} 

\newcommand{\dd}{\mathrm{d}}

\newcommand{\X}{\ensuremath{X_{0}}}

\newcommand{\tub}{\ensuremath{\mathrm{Tub} }}
\newcommand{\F}{\ensuremath{\mathcal{F}}}

\newcommand{\codim}{\ensuremath{\mathrm{codim}}}
\newcommand{\Sl}{\ensuremath{\mathcal{S}}}

\newcommand{\dank}{\textsf{Acknowledgments.\ }} 

\newcommand{\Iso}{\ensuremath{\mathrm{Iso}}}

\newcommand{\RR}{\mathbb R}
\newcommand{\N}{\mathbb N}

\newcommand{\metric}{\ensuremath{ \mathrm{g} }}

\newcommand{\minimalstratum}{\ensuremath{ \Sigma }}

\def\In{\subseteq}

\def\RR{\mathbb{R}}
\def\mc{\mathcal}

\def\H{\mc{H}}

\def\fol{\mc{F}}

\def\e{\varepsilon}
\newcommand{\satsub}{\ensuremath{ Y_{\F} }}
\newcommand{\newN}{\ensuremath{ \textsf{N}  }}
\newcommand{\FNNo}{\ensuremath{\mathcal{F}_{\newN'}^{0}}} 

\def\ddt{\frac{d}{dt}}
\def\ddtn{\frac{d^n} {dt^n}}

\newtheorem{conjecture}[theorem]{Conjecture}

\begin{document}


\title[Isometric flows on orbit spaces and Molino's conjecture]{Smoothness of isometric flows on orbit spaces  and applications to the theory of foliations}




\author{Marcos M. Alexandrino}

\author{Marco Radeschi}


\address{Marcos M. Alexandrino \hfill\break\indent 
Universidade de S\~{a}o Paulo\\
Instituto de Matem\'{a}tica e Estat\'{\i}stica, \hfill\break\indent
 Rua do Mat\~{a}o 1010,05508 090 S\~{a}o Paulo, Brazil}
\email{marcosmalex@yahoo.de, malex@ime.usp.br}

\address{ Marco Radeschi\hfill\break\indent 
Mathematisches Institut\\
 WWU M\"unster, Einsteinstr. 62, M\"unster, Germany.}
\email{mrade\_02@uni-muenster.de}

\thanks{The first author was  supported by CNPq and partially supported by FAPESP. The second author was partially supported by Benjamin Franklin Fellowship}



\begin{abstract}

We prove here that 
given a proper isometric action $K\times M\to M$ on a complete Riemannian manifold $M$ then 
every continuous isometric flow on  the orbit space $M/K$   is smooth, i.e., it is the projection of an $K$-equivariant smooth flow  on the manifold $M$. 
As a direct corollary we  infer the smoothness of isometric actions on orbit spaces.
Another relevant application of our result concerns Molino's conjecture, which states that
 the  partition of a Riemannian manifold into the closures of   the leaves of a singular Riemannian foliation is still a singular Riemannian foliation.
 We prove Molino's conjecture for the main class of foliations considered in his book, namely orbit-like foliations.

\end{abstract}



\maketitle

\setcounter{tocdepth}{1}
\tableofcontents

%
%

\section{Introduction}

Given a Riemannian manifold $M$ on which a compact Lie group $K$ acts by isometries, the quotient $M/K$ is in general not a manifold. Nevertheless, the canonical projection $\pi:M\to M/K$ gives $M/K$ the structure of a Hausdorff metric space. Moreover, following \cite{Schwarz}, one can define a ``smooth structure'' on $M/K$ to be the $\mathbb{R}$-algebra $C^{\infty}(M/K)$ consisting of functions $f:M/K\to \mathbb{R}$ whose pullback $\pi^*f$ is a smooth, $K$-invariant function on  $M$. If $M/K$ is a manifold, the smooth structure defined here corresponds to the more familiar notion of smooth structure.
A map $F:M/K\to M'/K'$ is called \emph{smooth} if the pull-back of a smooth function $f\in C^{\infty}(M'/K')$ is a smooth function $F^{*}f$ on $M/K$.

These concepts can actually be formulated in the wider context of singular Riemannian foliations (SRF for short). 
A singular foliation $\F$ is called \emph{Riemannian} if every geodesic perpendicular to one leaf is perpendicular to every leaf it meets. 
 The decomposition of a Riemannian manifold into the orbits of some isometric action is a special example of a singular Riemannian foliation, that is called  \emph{Riemannian homogeneous foliation}. Given a singular Riemannian foliation $(M,\F)$ with compact leaves,  one can define a quotient $M/\F$ and again endow it with a metric structure and a smooth structure, exactly as for group actions.

 In \cite[Corollary 2.4]{Schwarz}  Schwarz proved that given a proper  action $K\times M\to M$ each smooth flows on the orbit space $M/K$ is a 
 projection of an $K$-equivariant smooth flow  on the manifold $M$, and hence solved Bredon's \emph{Isotopy Lift Conjecture}; see details in 
 \cite{Schwarz}.

Our main result concerns with smoothness of continous flow of isometries (i.e., continuous 1-parameter groups of isometries) on orbit spaces.

\begin{theorem}
\label{main-theorem}
Let $M$ be a complete Riemannian manifold and $K \times M \to M$ a proper isometric action. 
Let 
$$
\varphi:M/K\times I\to M/K
$$
be a continuous flow of isometries on the orbit space.  Then $\varphi$ is smooth, and hence it is the projection of an $K$-equivariant flow on $M$.
\end{theorem}
\begin{remark}
Along the proof of the theorem, we will only assume that the action is a proper action and that the orbits of the action are leaves 
of a SRF $\F$ on  $M$. This mild generalization includes the case of proper isometric actions and is more suitable for the proof the theorem.  
\end{remark}

 The above result  implies the next corollary; see details in Section \ref{section-corollary}.

\begin{corollary}
\label{corollary}
Let $K\times M\to M$ be a proper isometric action on a complete Riemannian manifold $M$. Let $H$ be a  connected  Lie group acting by isometries on $M/K$. 
  Then  the action $H\times (M/K)\to M/K$ is smooth.
\end{corollary} 
\begin{remark}
 Recently, there have been some papers devoted to the study of isometries of singular spaces. 
In  \cite{ColdingNaber} Colding and  Naber proved that   isometry group of any, even collapsed, limit of manifolds with a uniform lower Ricci curvature bound is a Lie group.  In Proposition \ref{liegroupstructure} we will briefly discuss the particular case of group of isometries of a leaf space, proving that
 each connected compact  group of isometries of $M/\F$ is a Lie group.
In \cite{garciaGuijarro} Galaz-Garcia and  Guijarro   determined the maximal dimension of isometry group of Alexandrov spaces.  
In \cite{GorodskiLytchak} Gorodski and Lytchak investigated classes of orthogonal representations of compact Lie groups that 
have isometric orbit spaces.
Finally in \cite{AlexRadeschi}, we proved that if $\F$ is a closed SRF on $M$, then each isometry in the
identity component of the isometry group of $M/\F$ is a smooth map.

\end{remark}

Flows of isometries on the leaf spaces of  foliations appear naturally in the study of the dynamical behavior of \emph{non closed} singular Riemannian foliations. Recall that a (locally closed) singular Riemannian foliation $(M,\F)$ is locally described by  submetries $\pi_{\alpha}:U_{\alpha}\to U_{\alpha}/\F_{\alpha}$, where $\{U_{\alpha}\}$ is an open cover of $M$ and $\F_{\alpha}$ denotes the restriction of $\F$ to $U_{\alpha}$. If a leaf $L$ is not closed, one might be interested to understand how it intersects a given neighborhood $U_{\alpha}$, and in particular how the closure $\overline{L}$ of $L$ intersects $U_{\alpha}$. 
It turns out  that the projection $\pi_{\alpha}(\overline{L}\cap U_{\alpha})$ (that is contained in the local quotient of a stratum) is a submanifold, which is spanned by continuous flows of isometries $\varphi_{\alpha}$ on $U_{\alpha}/\F_{\alpha}$, cf. \cite[Thm 5.2]{Molino}. Therefore, in order to better understand the closure of $\overline{L}$, it would be relevant to understand if these flows admit  smooth lifts.

The above discussion already suggests  that Theorem \ref{main-theorem} should be a useful tool in the study of dynamical behavior of singular Riemannians foliations and should help to solve Molino's conjecture in important cases.

\begin{conjecture}[Molino]
Let $(M,\F)$ be a singular Riemannian foliation. Then the partition $\overline{\F}$ given by the closures of the leaves of $\F$ is again a singular Riemannian foliation.
\end{conjecture}

Molino himself proved the conjecture for regular Riemannian foliations, i.e., foliations where all the leaves have the same dimension; see \cite{Molino}. In \cite{Alex4}, the first author proved the conjecture for polar foliations, i.e., foliations admitting a totally geodesic submanifold transveral to the regular leaves and   which meets every leaf perpendicularly. Recently, the first author and Lytchak remarked in \cite{AlexLytchak} that  Molino's conjecture holds for so-called \emph{infinitesimally polar foliations} which, locally around each point $p\in M$, are foliated diffeomorphic to  a polar foliation. This is equivalent to saying that any local quotient $U_{\alpha}/\F_{\alpha}$ is an orbifold; see also \cite{AlexBriquetToeben}.

In this paper we prove Molino's conjecture for the  class of singular foliations considered in his book, namely orbit-like foliations; see \cite[p. 210]{Molino} for Molino's description about the state of the art of the known foliations at that time.
Recall that a singular Riemannian foliation is called  
\emph{orbit-like foliation}, if its restriction to each slice is diffeomorphic to a homogeneous foliation; see Section \ref{Preliminaries} for definitions, examples and remarks.

\begin{theorem}
\label{theorem-molinos-conjecture-orbitlike} 
Let $\F$ be an  orbit-like foliation on a complete Riemannian manifold $M.$ Then the closure of the leaves of $\F$ is a singular Riemannian foliation. 
\end{theorem}
\begin{remark}
As we will see in Corollary \ref{corollary-Forbitlike-homogeneous}, if $\F$ is a closed orbit-like foliation, then for each point of the leaf space $M/\F$
one can find a neighborhood that can be identified with $(\newN/G)/\H$ where $G$ is a compact group acting on a submanifold $\newN$ of $M$ and $\H$
is a pseudogroup of isometries acting on $\newN/G$. This provides a local description of  $M/\F$ and its smooth structure. This kind of result may be interesting in the study of proper groupoids and integrable Poisson manifolds
\cite{CrainicFernandes}. According to \cite{CrainicIvan}
and \cite{Pflaum}, the orbits of the proper groupoids are,
at least locally, described as  leaves (plaques) of orbit-like foliations for the apropriate metric. In particular, when $M$ is compact, they can be seen as leaves  of orbit-like foliation on a compact manifold. On the other hand, recall that the orbits of a proper groupoid are closed. Since there exist orbit-like foliations with non closed leaves, they are  not orbits of  proper groupoids.

\end{remark}

This paper is organized as follows.
In Section \ref{sketch} we briefly sketch the proofs of Theorems  \ref{main-theorem} and \ref{theorem-molinos-conjecture-orbitlike}.
In Section \ref{Preliminaries} we review, based on \cite{AlexRadeschi}, some basic concepts such as singular Riemannian foliations, infinitesimal foliations, orbit-like foliations and flows of isometries on  leaf spaces.
In Section \ref{Preliminaries} we also provide a quick introduction to the new tools, namely \emph{blow-up functions} and \emph{ reduction of a foliation} that are used in the proof of the main theorems.
In Section \ref{section proof-smoothflow} and Section \ref{sec-proof-molino}  we prove Theorem \ref{main-theorem} and Theorem \ref{theorem-molinos-conjecture-orbitlike} respectively. Finally in Section \ref{new-tools} we give the proofs of the new tools.

\dank The authors are very grateful to  Alexander Lytchak for inspiring the main questions of this work and  for his insightful comments and suggestions.  
The authors also thank Gudlaugur Thorbergsson and Burkhard Wilking for their support and Stephan Wiesendorf and Dirk T\"{o}ben for their  suggestions. 

%
%

\section{Sketch of proof of the main results}
\label{sketch}

For the sake of motivation, in this section we provide some ideas of the proofs of the main results. 
Technical aspects of the proofs will be discussed in later sections of the paper. 

\subsection{Sketch of proof of Theorem \ref{main-theorem}}

\subsubsection{Euclidean case}
\label{subsection-sketch-euclideancase}

In what follows  we briefly give an idea of proof of Theorem \ref{main-theorem}  in a particular but important case.

We assume that:
\begin{enumerate}
\item $M=\mathbb{R}^{n}$  

\item $K$ is a closed subgroup of $\mathrm{Iso}(\mathbb{R}^{n})$.

\item The foliation is the partition of the ambient into the orbits, i.e., $\F:=\{K(p)\}_{p\in \mathbb{R}^{n}}$

\item there exists a $C^1$-flow $\varphi$ on $\mathbb{R}^n/K$,

\item Theorem \ref{main-theorem} is already valid for foliations with lower singularities.

\end{enumerate}

We want to prove that the flow $\varphi$ is smooth.
This is equivalent to saying that for each $K$-invariant smooth function $h:\mathbb{R}^{n}\to\RR$ the
function $u_{h}: \mathbb{R}^n\times I\to \mathbb{R}$ 
defined as $u_h(x,t):=(\varphi_{t}^{*}h)(x)$ is smooth.

Firstly, recall that  the mean curvature vector field $\vec{H}$ of the principal $K$-orbits  projects to a well defined vector field  
on $\RR^n/K$  and its projection is preserved by every isometry $\varphi_{t}$; see \cite{AlexRadeschi}.

Secondly, note that isometry $\varphi_{t}$ also preserves the Laplacian operator on the orbit space of the principal strata.

Finally, recall that the Laplacian $\bigtriangleup$ of the principal stratum $(\mathbb{R}^{n})_{princ.}$ 
can be obtained  from the Laplacian  $\underline{\bigtriangleup}=\bigtriangleup_{(\mathbb{R}^{n}/K)_{princ.}}$ on its orbit space $(\mathbb{R}^{n}/K)_{princ.}$ by
the following equation:
$$\bigtriangleup h= \underline{\bigtriangleup} h - \langle \nabla h, \vec{H}\rangle.$$ 

Since $\varphi_{t}$ preserves both summands in the right-hand side of the above equation, we infer that 
 $\bigtriangleup \varphi_{t}^{*} h  =\varphi_{t}^{*} \bigtriangleup h$ holds in the principal stratum.

Therefore, as remarked in  \cite{AlexRadeschi}, the following equation holds in a weak sense
\begin{equation*}
\bigtriangleup \varphi_{t}^{*} h  =\varphi_{t}^{*} \bigtriangleup h,
\end{equation*} 
where $\bigtriangleup h$ denotes the Laplacian operator.

 Set $u_{h}(\cdot,t):=\varphi_{t}^{*} h$ and  $f(\cdot,t)=-\varphi_{t}^{*}(t)\bigtriangleup h +\frac{d}{d t} u_{h}(\cdot,t)$.
From the equation above the following equation holds in a weak sense
\begin{equation*}
\frac{d}{d t} u_{h}-\bigtriangleup u_{h}=f.
\end{equation*}

This equation motivates us to consider the  results of regularity of parabolic equations to prove that $u_{h}$ is smooth.
The regularity theory of parabolic equations requires some compatibility conditions. These conditions  can be checked using 
the assumption that Theorem \ref{main-theorem} is already valid for foliations with lower singularities, blow-up of $\F$
and some properties of a class of functions that we have called  ``blow-up functions''. 
Once these conditions have been checked, 
we can apply the regularity theory and a bootstrap type argument to conclude that 
$u_h$ is smooth.


\subsubsection{General case}
Let us now say a few words about the proof of the theorem in the general case. 
As we will see later, it will be possible to reduce our problem or to the Euclidean case or to the folowing situation:
\begin{enumerate}
\item $\F=\{K(p)\}_{p\in N}$ is an homogenous SRF (i.e, the leaves are orbits of a compact group $K$) on a trivial fiber bundle $N\to Y$, where  $N=X\times Y$ and $Y=\mathbb{R}$is an integral curve of $\varphi$.
\item there is a metric $\metric_{N}$ such that the fibers $X_{t}:= X\times \{t\}$  are flat but not necessarly totally geodesic.
\item $\varphi$ may not be a flow of  isometries in the quotient $(N,\metric_{N})/\F$ but, for each fixed $t$, the flow  $\varphi$ induces an  an isometry 
$$\phi(t):X_0/\F:=(X\times \{0\}) /\F \to X_t/\F:=(X\times\{ t\})/ \F$$ defined as $\phi(t)(x^{*}):= \varphi(x^{*},t)$. 

\end{enumerate}

We will divide the proof of the smoothness of $\varphi$
into two steps. 

In the first step, we fix the fiber 
$X_0:=X\times\{0\}$ and prove the smoothness on $X_0\times I.$
In other words, for each \emph{basic} function $h$, i.e., a $K$-invariant function,
we prove that the map $(x,0,t)\mapsto \varphi^{*}h(x,0,t)$ is smooth. This is the most important step. Set $u_{h}(\cdot,t):=\phi(t)^{*} h$ and  $f(\cdot,t)=-\phi^{*}(t)\bigtriangleup h +\frac{d}{d t} u_{h}(\cdot,t)$.
Here $\bigtriangleup$ is the Laplacian on the fiber. 
Since the fibers are flat, one can follow the same argumets used in the proof of the Euclidian case and get the next equation in weak sense
\begin{equation*}
\frac{d}{d t} u_{h}-\bigtriangleup u_{h}=f.
\end{equation*}
Again, like in the Euclidean case
we can apply the regularity theory, blow-ups and a bootstrap type argument to conclude that 
$u_h$ is smooth on $X_0\times I$, i.e., the map $(x,0,t)\mapsto \varphi^{*}h (x,0,t)$ is smooth.

In the second step of the proof, we extend the smoothness of $\varphi$
to the whole $(X\times Y)/\F\times I$. In other words,
we prove that  the map $(x,y,t)\mapsto \varphi^{*}h (x,y,t)$ is smooth for each basic function $h$.
This is done using the inverse function theorem on
orbit spaces \cite{Schwarz} and the fact that $\varphi$ is a flow.


\subsection{Sketch of proof of Theorem \ref{theorem-molinos-conjecture-orbitlike} }

We first need some definitions and notations. Let
\begin{enumerate}
\item $\Sigma$ be a minimal stratum,
\item $D\subset \Sigma$ be a slice at $q\in \Sigma$ to the foliation $\F|_\Sigma$,
\item $\newN:=\exp(\nu\Sigma|_{D})\cap B_{\epsilon}(D)$ where   $\nu\Sigma|_{D}$ denotes   the restriction of the normal bundle of $\Sigma$ to  the slice $D$,
\item  $\FNo:=\F \cap \newN$ be the foliation induced on $\newN$.  
\end{enumerate}

The difficult part of the proof of Molino's conjecture is to prove that the closure $\overline{\F}$ of $\F$ is a \emph{singular foliation}; see Definition \ref{definition-srf} below.
More precisely, our goal is to prove that  for each given  vector $v\in T_{q}D$  tangent  to  $\overline{L_{q}} \subset \Sigma$,  there exists a smooth vector field $\vec{Y}$ tangent to the leaves of $\overline{\F}$ so that $\vec{Y}(q)=v$.

Using singular holonomy and desingularization we prove that there exists a continuous flow of isometries $\varphi$ on 
$(\newN,\tilde{\mathrm{g}})/\FNo$ such that ${\pi_{\FNo}}_{*}v$ is tangent to the integral curve of $\varphi$. Here $\pi_{\FNo}$ is the canonical projection and   $\tilde{\mathrm{g}}$ is a 
metric on $\newN$ so that $(\newN,\FNo)$ is a singular Riemannian foliation with the same \emph{transverse metric} of $\F$, i.e., the distance between  leaves of $\FNo$ is the
 same as the distance between the plaques of $\F$ that contain such leaves.

Since  $\FNo$ is homogeneous, we can apply Theorem 
\ref{main-theorem} and conclude that $\varphi$ is smooth. The existence of $\vec{Y}$ follows from \emph{Schwarz's Theorem} \cite{Schwarz} that allows us to lift
smooth flows on the orbit space to smooth flows on the manifold.

%
%
\bigskip

\section{Preliminaries}
\label{Preliminaries}

\subsection{Singular Riemannian foliations}
Let us recall the definition of a singular Riemannian foliation.
\begin{definition}[SRF]
\label{definition-srf}
A partition $\F$ of a complete Riemannian manifold $M$ by connected immersed submanifolds (the \emph{leaves}) is called a  \emph{singular Riemannian foliation} (SRF for short)  if it satisfies condition (a) and (b):
\begin{enumerate}
\item[(a)] $\F$ is a \emph{singular foliation},
i.e., for each leaf $L$ and each $v\in TL$ with footpoint $p,$ there is a smooth vector field $\vec{Y}$ with $\vec{Y}(p)=v$ that belongs to $\mathfrak{X}(\F)$, i.e.,  that is tangent at each point to the corresponding leaf.
\item[(b)] $\F$ is a \emph{transnormal system}, i.e.,  every geodesic  perpendicular  to one leaf   is perpendicular to every leaf it meets.
\end{enumerate}
\end{definition}
A leaf $L$ of  $\F$ (and each point in $L$) is called \emph{regular} if the dimension of $L$ is maximal, otherwise $L$ is called {\it singular}. In addition a regular leaf is called \emph{principal} if it has trivial holonomy; for a definition of holonomy, see e.g. \cite[page 22]{Molino}.

A typical example of a singular Riemannian foliation is  the partition of a Riemannian manifold into the connected components of the orbits of an isometric action. Such singular Riemannian foliations are called \emph{Riemannian homogeneous}. In this case the principal leaves coincide with the principal orbits. We will sometimes denote a Riemannian homogeneous foliation, given by the action of a Lie group $K$, by $(M, K)$, provided the $K$-action is understood.

If a singular Riemannian foliation $(M,\F)$ is spanned by a smooth action of a Lie group, which does not necessarily act by isometries, then we call such a foliation \emph{homogeneous}.

\medskip

\subsection{The infinitesimal foliation at a point}
Let $(M,\F)$ be a singular Riemannian foliation. Given a point $p\in M$ and some small $\epsilon>0$, let $S_p=\exp_p (\nu_pL_p) \cap B_{\epsilon}(p)$ be a \emph{slice} at $p$, where $B_{\epsilon}(p)$ is the distance ball of radius $\epsilon$ around $p$. The foliation $\F$ induces a foliation $\F|_{S_p}$ on $S_p$ by letting the leaves of $\F|_{S_p}$ be the connected components of the intersection between $S_p$ and the leaves of $\F$. In general the foliation $(S_p,\F|_{S_p})$ is not a singular Riemannian foliation with respect to the induced metric on $S_p$. Nevertheless, the \emph{pull-back} foliation $\exp_p^{*}(\F)$ is a singular Riemannian foliation on $\nu_pL_p\cap B_{\epsilon}(0)$ equipped with the Euclidean metric (cf.~\cite[Proposition 6.5]{Molino}), and it is invariant under homotheties fixing the origin (cf.~\cite[Lemma 6.2]{Molino}). In particular, it is possible to extend $\exp^*(\F)$ to all of $\nu_pL_p$, giving rise to a singular Riemannian foliation $(\nu_pL_p,\F_p)$ called the \emph{infinitesimal foliation} of $\F$ at $p$.

If $(M,K)$ is Riemannian homogeneous, the infinitesimal foliation $(\nu_pL_p,\F_p)$ is again Riemannian homogeneous, given by the action of (the identity component of) the isotropy group $K_p^0$ on $\nu_pL_p$ (the \emph{slice representation}).

The converse however is not true: namely, there are examples of non-Riemannian homogeneous foliation all of whose infinitesimal foliations are.
\begin{definition}
\label{definition-locally-homogeneous}
A SRF $\F$ on $M$ is called an \emph{orbit-like foliation} if for each point $q$ there exits a compact group $K_q$ of isometries of $\nu_{q}L_{q}$ such that the infinitesimal foliation $\F_{q}$ is the partition of $\nu_{q}L_{q}$ into the orbits of the action of $K_{q}$. 
\end{definition}

Examples of orbit-like foliations are given by the closures of (regular) Riemannian foliations. Other examples can be obtained via a procedure called \emph{suspension of homomorphism}; for more details see  e.g. \cite[sec. 3.7]{Molino}.

\begin{example}
\label{example-suspension-homomorphism}

Consider a nonhomogeneous manifold  $Q$ and let $(V=V_1\times V_2, \widetilde{\F})$ be a homogeneous Riemannian foliation given by the orbits of an isometric action of a closed subgroup $K\In\Iso(V_1)$ that fixes a point in $V_1$. Take a homomorphism $\rho:\pi_{1}(Q, q_{0})\rightarrow  H\subset \Iso(V_2)$. By construction, $\rho$ induces an action $\rho'=\rho\times id: \pi_1(Q,q_0)\to\Iso( V_1\times V_2)$ that preserves the foliation $\widetilde{\F}$. Then $\pi_{1}(Q,q_{0})$ acts diagonally on the product $\widetilde{M}=\widetilde{Q}\times V$ where $\widetilde{Q}$ is the universal cover of $Q$ and the action on the first factor is the action of deck transformations. Since the action is free, the quotient  $M:= \widetilde{M}/\pi_1(Q,q_0)$ is a manifold, and there is a map
\begin{eqnarray*}
\mathcal{P}: M &\longrightarrow& Q \\
{}[\widetilde{q},v]& \longmapsto & [\widetilde{q}]
\end{eqnarray*}
that makes $M$ the total space of a fiber bundle with base $Q$. In addition, the fiber is $V$ and the structural group is given by the image of $\rho'$.
Finally, the SRF $(\widetilde{M}, \widetilde{Q}\times \widetilde{\F})$ induces via $\widetilde{M}\to M$ a SRF $(M,\F)$, which can be proved to be an
 (nonhomogeneous) orbit-like foliation.
\end{example}


\medskip

\subsection{Smooth structure of a leaf space.} 
Let $(M,\F)$ be a SRF with closed leaves. The quotient $M/\F$ is equipped with the natural  quotient metric and a natural quotient ``$C^{k}$ structure''. The $C^{k}$ \emph{structure} on $M/\F$ is  given by
the sheaf $C^k(M/\F):=C^k_b(M,\F)$ of $C^{k}$ \emph{basic functions} on $M$, i.e., those functions that are constant along the leaves of $\F$. A function $f\in C^k_b(M,\F)$ is called \emph{of class $C^k$}, while $f\in C^{\infty}(M,\fol)$ is called \emph{smooth}.
One says that a map $\varphi: M_{1}/\F_{1} \to M_{2} /\F_{2}$  between two leaf spaces of SRF  is of class $C^{k}$ if the pull-back of a smooth function $f\in C^{\infty}(M_{2}/\F_{2})$ is a function $\varphi^*f\in C^k(M_{1}/\F_{1})$. When $\varphi$ is smooth, this definition coincides with the definition of  Schwarz \cite{Schwarz}.


\medskip

\subsection{Flow in a leaf space}
Consider a SRF $\F$ on a complete Riemannian manifold $M$ with closed leaves. 

\begin{definition}
\label{definition-flow}
A continuous map  $\varphi:M/\F\times I\to M/\F$ is called \emph{a flow on} $M/\F$ if the following conditions hold:
\begin{enumerate}
\item[(a)] $\varphi$ is a one-parameter local group;
\item[(b)] for each $p\in M/\F$ each integral curve $t\to \varphi(p,t)$ is contained in the quotient of a stratum;
\item[(c)] there exists a  locally bounded derivative $\bar{Y}$ on $M/\F$  associated to the flow, such that for each smooth basic function $h$ we have
$\bar{Y}\cdot h (p)=\frac{d}{d t}h(\varphi(p,t))|_{t=0}$. 
\end{enumerate}
\end{definition}

\begin{remark}
\label{remark-one-parametergroup}
These conditions are independent and do not imply each other. However, in the particular case where $\varphi$ is a one-parameter local group of isometries of $M/\F$, one can prove that (b) holds (see e.g., \cite[section 5.1]{GorodskiLytchak}) and that (a) and (b) imply (c); see e.g. 
Remark \ref{remark-locally bounded}.

\end{remark}

\medskip

\subsection{SRF's with disconnected leaves}
\label{subsection-disconnected-leaves}
Sometimes one has to consider Riemannian foliations with non-connected leaves. This kind of foliations come up naturally: consider for example a Riemannian homogeneous foliation $(N,K)$. Even when $K$ itself is connected, some isotropy subgroup $K_p$ might not be, and its orbits under the slice representation might also be disconnected. Therefore the Riemannian homogeneous foliation $(\nu_p(K\cdot p),K_p)$ would be an example of a disconnected singular Riemannian foliation. In general, a \emph{singular Riemannian foliation with disconnected leaves} $(N,\F)$ is a triple $(N,\F^{0},\H)$ where $(N,\F^{0})$ is a (usual) SRF, $\H$ is a group of isometries of $N/\F^{0}$, and the \emph{non-connected leaves} of $\F$ are just the orbits $\H\cdot L_p$, for $L_p\in \F^{0}$.
\\

A slight generalization of such a triple, which we still call singular Riemannian foliation with disconnected leaves and still denote $(N,\F)$, is a triple 
$(N,\F^{0},\H)$ where $(N,\F^{0})$ is again a SRF, and $\H$ is a pseudogroup of local isometries of $N/\F^{0}$. Again, the non-connected leaves of $\F$ are $\H$-orbits of leaves of $\F^{0}$. Such a foliation appears naturally when dealing with the \emph{singular holonomy} around a non-closed leaf (cf. Definition \ref{singular-holonomy-pseudogroup}).


\subsection{A quick introduction to the new tools}
\label{section-quickintroduction-tools}

In this section we provide 
a quick introduction to the new tools used in the proofs of
the  Theorems \ref{main-theorem} and  
\ref{theorem-molinos-conjecture-orbitlike}.
This will allow the reader to follow the proofs of the main theorems.
Details and proofs about the new tools can be found in  Section \ref{new-tools}.

\subsubsection{Blow-up's}
\label{new-section-blow-up}

Throughout this section, $\Sigma$ will denote a relatively compact  neighborhood in the minimal stratum of the foliation $\F$ on $M$. 
If the minimal stratum is compact, one can, equivalently, take $\Sigma$ to be the whole minimal stratum.

Following a procedure analogous to the classical blow-up of isometric actions, one can prove
that there exist:
\begin{enumerate}

\item a saturated  neighborhood $B$ of $\Sigma$ and a new  Riemannian manifold $\widehat{B}$ with a (blow-up) metric $\hat{\metric}$;

\item a new (singular) Riemannian foliation $\widehat{\F}$ on $(\widehat{B},\hat{\metric})$ whose  minimal singular leaves
have dimension greater than the dimension of the minimal singular leaves of $\F$;

\item and a (blow-up) map $\widehat{\rho}:(\widehat{B},\widehat{\F})\to (B,\F)$
that is foliated, i.e., that sends leaves to leaves. 

\end{enumerate}

\begin{remark}
\label{remark-new-desingularization}
As we will recall in Section \ref{section-blow-up}, it is possible to construct, via compositions of projective blow-up's, a surjective map $\rho_{\epsilon} : M_{\epsilon}\to  M$ with the following properties:
\begin{enumerate}
\item $M_{\epsilon}$ is a smooth complete Riemannian manifold foliated by a regular Riemannian foliation $\F_{\epsilon}$.
\item The map  $\rho_{\epsilon}$ sends leaves of $\F_{\epsilon}$ to leaves of $\F$.
\end{enumerate}
This map is called a \emph{desingularization map}.
 If $M/\F$ is compact, then for each small $\epsilon>0$ one can choose $M_{\epsilon}$ and $\rho_{\epsilon}$ so that $d_{GH}(M_{\epsilon}/ \F_{\epsilon}, M/\F)<\epsilon$ where $d_{GH}$ is the Gromov-Hausdorff distance. 
\end{remark}

The first useful  property of blow-up's  will be proved in 
Proposition \ref{lemma-lifted-flow-isometry}, namely:
\begin{proposition}
\label{new-lemma-lifted-flow-isometry}
Each flow of isometries $\varphi:B/\F\times I\to B/\F$ can be lifted to a  flow of isometries 
$\hat{\varphi}: \widehat{B} /\widehat{\F} \times I \to \widehat{B} /\widehat{\F}.$ 
\end{proposition}

It is natural to ask: \emph{under what conditions
does the smoothness of the pull back function $\hat{\rho}^{*}f$ of a basic
function $f$ imply the differentiability of $f$ on $B$?}
The answer to this question
 in particular cases will be  useful to check regular conditions of certain
(weak) parabolic equations.

As we will see in Definition \ref{def-blow-upfunction}, a continuous $\F$-basic function $h$ on $B$ belongs to $\mathcal{B}$ or is a \emph{blow-up function}, if
\begin{enumerate}
\item[(a)] $h\circ\hat{\rho}$ is a smooth $\widehat{\F}$-basic function on $\widehat{B}$.
\item[(b)]  
The restriction of $h$ to $\Sigma$ and to $B-\Sigma$ is smooth.
\item[(c)] $X\cdot h=0$ for each $X\in \nu\Sigma$.
\end{enumerate} 

The answer to our question will be given in Proposition \ref{lemma-tecnhical}, namely:
\begin{proposition}
\label{new-lemma-tecnhical}
If $h\in \mathcal{B}$ then  $h$ is a $C^{1}$ function. 
\end{proposition}

The family of blow-up functions has another relevant property. It is closed under derivation of transverse Killing
vector fields. More precisely, consider a flow of of isometries $\varphi$ on $B/\F$ and 
let $\bar{Y}$ be the associated derivative in the quotient $B/\F$; recall Definition \ref{definition-flow}. 
Then, we have the next result proved  in Proposition
\ref{B-are-X-closed}:

\begin{proposition}
\label{new-B-are-X-closed}
Assume that the lift $\hat{\varphi}$ on  $\widehat{B}/\widehat{\F}$ is smooth. 
Then  for each $h\in \mathcal{B}$ we have  $\varphi_{s}^{*} h$ and $\varphi_{s}^{*}(\bar{Y}\cdot h) \in \mathcal{B}$.
\end{proposition}

As we have seen in Section \ref{subsection-sketch-euclideancase}, 
we will have to deal with functions $u_h(x,t):=(\varphi_{t}^{*}h)(x)$.  
Propositions \ref{new-lemma-tecnhical} and \ref{new-B-are-X-closed}
assure that $\frac{d^{n}}{d t^{n}} u_{h}(\cdot, t)$ is $C^{1}$ for all $t$. 
More details will be present in the proof of Lemma \ref{time-derivative-u-C1}.

\subsubsection{Local reduction}
\label{new-subsection-local-reduction}

In Section \ref{sketch} we have briefly seen the appearance of the local reduction $\newN$ and the metrics $\tilde{\mathrm{g}}$ 
and $\mathrm{g}_\newN$. In what follows we are going to say a few more words about them.

We start by fixing some notations:
\begin{enumerate}
\item Let $(M,\F)$ be a SRF, and let $\Sigma\In M$ be a stratum of $\F$. 
\item Let $Y$ be a submanifold contained in a slice (a transverse submanifold) of the regular foliation $(\Sigma,\F|_{\Sigma})$. 
\item  Consider $\satsub:=\pi_{\F}^{-1}(\pi_{\F}(Y))$ the \emph{saturation} of $Y$. We also assume that $Y$ coincides with the intersection of $\satsub$  
with the slice.
\item Let $\nu(\satsub)\big|_Y$ be the normal bundle of the saturation $\satsub$ restricted to $Y$.
\end{enumerate}

We  define the \emph{local reduction of $(M,\F)$ along $Y$} as
$$\newN:= \exp\left(\nu(\satsub)\big|_Y\right)\cap B_\e\satsub,$$ 
where  $B_\e\satsub$ is a neighborhood of $\satsub$.
The local reduction $\newN$ is a fiber bundle, where the projection is the metric projection $\Projection_Y:\newN\to Y$. The fiber of $\Projection_Y$ at a point $p\in Y$ will be denoted by $\newN_p$.

\begin{remark}
Notice that $\dim Y\leq\dim(\Sigma/\F)$. Throughout the paper we will be especially interested in the cases $\dim Y=1$ and $\dim Y=\dim(\Sigma/\F)$. More precisely, in the proof of Theorem \ref{main-theorem} we will take $Y$ to be a curve that projects to an integral curve of the flow $\varphi$, while for Theorem \ref{theorem-molinos-conjecture-orbitlike} we will consider $Y$ to be a slice of $(\Sigma,\F_{\Sigma})$.
\end{remark}

The foliation $\F$ intersects $\newN$ in a foliation $\FNo:=\F \cap \newN$.  
It is possible to check that the leaves of $\FNo$ are contained in the fibers of $\Projection_Y$.
The foliation $\FNo$ turns out to be a SRF with respect to 2 metrics.
 
The first metric  $\tilde{\mathrm{g}}$, constructed in Proposition \ref{lemma-same-transverse-metricN}, 
will be used in Section \ref{sec-proof-molino} to prove Theorem \ref{theorem-molinos-conjecture-orbitlike}.
\begin{proposition}
\label{new-lemma-same-transverse-metricN}
There exists a metric $\tilde{\mathrm{g}}$ on $\newN$ that preserves the transverse metric of $\F$, i.e., 
the distance between  leaves of $\FNo$ is the
 same as the distance between the plaques of $\F$ that contain such leaves.
In particular $\FNo$ is a  SRF on $(\newN,\tilde{\mathrm{g}})$.
\end{proposition}

As we will see in Corollary \ref{corollary-Forbitlike-homogeneous},  when $\F$ is an orbit-like foliation, then  $\FNo$ is homogenous, but not necessarly Riemannian homogenous.
\begin{proposition}
\label{new-corollary-Forbitlike-homogeneous}
Let $\F$ be an orbit-like foliation. Let $q$ be a singular point, $Y$ a slice in the stratification $\Sigma$ that contains $q$ and $\newN$
the reduction along $Y$. Then there exists a compact group $G$ acting on $\newN$ such that the leaves of  $\FNo$ are orbits of $G$, i.e., 
$\FNo$ is homogenous SRF on $(\newN,\tilde{\metric}).$

\end{proposition}

The second metric  $\mathrm{g}_\newN$, constructed in  Proposition \ref{lemma-submersion-N}, 
will be used in Section  \ref{section proof-smoothflow} to prove Theorem \ref{main-theorem}.

\begin{proposition}
\label{new-lemma-submersion-N}
There exists a metric $\mathrm{g}_\newN$ on $\newN$ with following properties:
\begin{enumerate}
\item[(a)] The submersion $\Projection_Y: \newN\to Y$ is Riemannian.
\item[(b)] Each fiber $\newN_p$ is flat.
\item[(c)] The foliation $\FNo$ is a SRF on $(\newN,\mathrm{g}_\newN).$
\end{enumerate}
\end{proposition}

Note that two different leaves $L^{0}_{1}$, $L^{0}_{2}$  of  $\FNo$ can be contained in the same leaf of $\F$. 
One can prove that for any such $L^{0}_{1}$, $L^{0}_{2}$, there exists an isometry $\tau$  on $\newN/\FNo$ (with respect to $\tilde{\metric}$) 
that sends  $L^{0}_{1}$ to  $L^{0}_{2}$, i.e.,
\begin{equation}
\label{new-definition-tau}
\tau(x_{1}^{*})=x_{2}^{*}
\end{equation}
where $x_{i}^{*}$ is the projection of $L_{i}^{0}$ in $\newN/\FNo$.
The pseudogroup generated by these isometries will be denoted by $\H$  and the  foliation with disconnected leaves $(N,\FNo,\H)$ will be denoted by $\FN$; see
 Definition \ref{singular-holonomy-pseudogroup}. 
The basic functions of $\FN$ coincide with basic functions of $\F$ (see Remark \ref{remark-reduction}) and hence, $\FN$ is the correct foliation to consider in the proof of 
 Theorem \ref{main-theorem}.

The next result will follow from Remark \ref{remark-property-gN} and Lemma \ref{lemma-thegroupG}.

\begin{proposition}
\label{new-remark-property-gN} 
Assume that $Y$ projects to an integral curve of a flow of isometries $\varphi$ on $M/\F$, where $\F$ is homogenous. 
Let $\newN$ be a local reduction along $Y$. For $Y$ small enough, the projection $\newN \to Y$ is trivial and therefore we can
identify $\newN$ with $X\times Y$ where $X$ is a fiber of $N_0$ and $Y=(-\epsilon,\epsilon)$.
The flow $\varphi$ may not be a flow of  isometries in the quotient $(N,\metric_{N})/\FN$ but, for each fixed $t$, the flow  $\varphi$ induces an  an isometry 
$$\phi(t):X_0/\FN:=(X\times \{0\}) /\FN \to X_t/\FN:=(X\times\{ t\})/ \FN$$ defined as $\phi(t)(x^{*}):= \varphi(x^{*},t)$. 
\end{proposition}


\section{Isometric flows on orbit spaces: proof of Theorem \ref{main-theorem}}
\label{section proof-smoothflow}

The goal of this section is to prove Theorem \ref{main-theorem}. In particular, we are assuming that the leaves of the SRF $(M,\F)$ are the orbits of a smooth proper action of a Lie group $K$ on $M$.

 In order to avoid cumbersome notations, we will denote every basic function on $M$ and the induced function on $M/\F$ by the same letter.  

We know that  Theorem \ref{main-theorem} is true when all orbits have the same dimension, see e.g. \cite[Salem appendix D]{Molino} and Swartz\cite{Swartz}.
We assume by induction that  Theorem \ref{main-theorem} is true for a number of strata lower or equal to  $n-1$. 
\\

Let $q$ be a point of the minimal stratum, $q^*\in M/\F$ its projection in the quotient, and $Y^*$ a neighborhood of $q^*$ in the orbit of $\varphi$ through $q^*$. The preimage $\satsub=\pi_\F^{-1}(Y^*)$ is a regularly saturated submanifold of $M$ (cf. Section \ref{new-subsection-local-reduction}) such that $\F|_{\satsub}$ has codimension 1. Let $Y$ be the intersection of $\satsub$ with a slice at $q$ for the action of $K$ on $M$, and $\newN$ the local reduction of $(M,\F)$ along $Y$;  cf. Section \ref{new-subsection-local-reduction}. Unless explicitly stated otherwise, we will always consider the Riemannian metric $\mathrm{g}_N$  on $\newN$ defined in Proposition \ref{new-lemma-submersion-N}.

Before we go through the details, let us briefly recall the main idea of the proof. It is enough to show that $\varphi$ is smooth on $\newN/\FN\times I$. In other words, for a given smooth $\FN$-basic function $h$ on $\newN$ we will prove that $\varphi^{*}h$ is a smooth basic function on $\newN\times I$ with respect to the foliation $\FN\times\{*\}=\{L_{\newN}\times\{*\}\}$.

We will divide the proof of the smoothness of $\varphi$ into two steps. 
\begin{itemize}
\item[Step 1)] We restrict our attention to a fiber $\X:=\newN_{q_0}$ of the metric projection $\Projection_Y:\newN\to Y$ (cf. Section \ref{new-subsection-local-reduction}), and we prove in Proposition \ref{proposition-varhi-smooth-onfiber} that the restriction of $\varphi^{*} h$  to $\X\times I$ is smooth.
The main idea here is to use some some arguments of \cite{AlexRadeschi} to check that $u_h(x,t):=\varphi^{*} h(x,t)$ is a weak solution of a parabolic equation; see equation \eqref{eq-weak-parabolic-solution}.
We apply regularity theory of solutions of  linear parabolic equations to prove that $u_h$ is smooth. 
This theory requires some initial regularity conditions that will be checked using Propositions \ref{new-lemma-tecnhical} and \ref{new-B-are-X-closed}; see details in Lemma \ref{time-derivative-u-C1}.

\item[Step 2)] We extend the smoothness of $\varphi$ to the whole $\newN/\FN\times I$ using the inverse function theorem for orbit spaces; see \cite[page 45]{Schwarz}. This is proved in Proposition \ref{proposition-varhi-smooth-onsubbundle}.
\end{itemize}

Since we will be working on $(\newN,\FN)$ instead of $(M,\F)$, we should better make sure that $(\newN,\FN)$ is still homogeneous.

\begin{lemma}[The group $G$]
\label{lemma-thegroupG}
The points on the curve $Y$ have the same isotropy group  $G:=K_q$. Moreover, the restriction of $\FN$ to $\newN$ is the partition of $\newN$ into the orbits of the action of $G$. 
\end{lemma}
\begin{proof}
Consider $D$ a slice at $q$ in the singular stratum $\Sigma$ of the restricted foliation $\F|_{\Sigma}$. 
Let us denote $\tilde{\varphi}_t$ a flow of isometries on $D$ which is a lift of $\varphi_t$ and so that $Y$ is an integral curve; see e.g. \cite{Swartz}.
We want to prove that 
\begin{equation}
\label{eq1-lemma-thegroupG}
K_{\tilde{\varphi}_{t}(q)}=K_q.
\end{equation}
Consider the action $\mu:K_{q}\times D\to D$ and the induced homomorphism $\mu:K_{q}\to \mathrm{Iso}(D)$. Since we are dealing with isotropy groups, in order to prove
equation \eqref{eq1-lemma-thegroupG} it suffices to prove that
\begin{equation}
\label{eq2-lemma-thegroupG}
\mu(K_{\tilde{\varphi}_{t}(q)})=\mu(K_q).
\end{equation}

We first claim that $\tilde{\varphi}_{t} \mu(K_{q})\tilde{\varphi}_{t}^{-1}= \mu(K_{\tilde{\varphi}_{t}(q)}).$ Let $p$ be a principal point in $D$
(i.e., the leaf $L_p$ has trivial holonomy in $\Sigma$) and consider  $k\in \mu(K_{q})$. Note that, since $p$ is principal, $kp\neq p$. Set $k_1:=\tilde{\varphi}_{t}k\tilde{\varphi}_{t}^{-1}$. Note that $k_1\tilde{\varphi}_{t}(p)=\tilde{\varphi}_t(kp)$. On the other hand, since $\varphi_t$ is  an isometry in the quotient, and in particular sends loops into loops, there exists a  $k_2\in \mu(K_{\tilde{\varphi}_{t}(q)})$ such that $k_2 \tilde{\varphi}_{t}(p)=\tilde{\varphi}_{t}(k p).$ Therefore, since the same argument applies to other principal points near  $p$ (recall that the set of principal points is an open and dense set) we infer that $k_1=k_2$ and hence  $\tilde{\varphi}_{t} \mu(K_{q})\tilde{\varphi}_{t}^{-1}\subset \mu(K_{\tilde{\varphi}_{t}(q)}).$ The proof of the other inclusion is identical and hence the claim has been proved.

Now, since $D$ a slice at $q$ of $\F|_{\Sigma}$, we have that $K_{\tilde{\varphi}_{t}(q)}, K_{q}$ are compact Lie subgroups and $K_{\tilde{\varphi}_{t}(q)}\subset K_{q}$. These facts and the above claim imply equation \eqref{eq2-lemma-thegroupG}.
\end{proof}

\begin{remark}
In the particular case where $\F$ is Riemannian homogeneous, one can check that  $G$ acts isometrically. 

\end{remark}


\medskip

\subsection{The first step} Let us fix a fiber $\X:=\newN_{q_0}$ of $\Projection_Y:\newN\to Y$, and denote $X_t:=\newN_{\tilde{\varphi}_t(q_0)}$. Since the flow $\varphi:(\newN,\tilde{\textrm{g}})/\FN\times I\to (\newN,\tilde{\textrm{g}})/\FN$ acts by isometries and $Y^*$ is an orbit, then by
Proposition \ref{new-remark-property-gN} each $\varphi_t$ is an isometry between $(\X,\mathrm{g}_\newN)/\FN\to( X_{t},\mathrm{g}_\newN)/\FN$, \emph{where the metrics on the fibers are now flat}.

Since $\X$ and $X_{t}$ are flat, it follows from \cite[Proposition 2.1]{AlexRadeschi} that the mean curvature vector fields of the leaves in $\X$ and $X_{t}$ project to well defined vector fields in the regular strata of $\X/\FN$, $X_t/\FN$ respectively, and moreover $\phi(t):=\varphi_t$ sends one vector field to the other. On the other hand since $\phi(t)$ is an isometry, it preserves the Laplacian operator in the principal part of $X_{t}/\FN$.

Set $u_h(\cdot,t):=\phi(t)^*h$. From what said above, we obtain that the following equation holds a in weak sense (cf.  \cite[Lemma 2.5]{AlexRadeschi}):
\begin{equation}
\label{eq-weak-laplacianterm}
\bigtriangleup u_h=\bigtriangleup \phi(t)^*h  =\phi(t)^* \bigtriangleup h=u_{\bigtriangleup h}
\end{equation} 
where $\bigtriangleup h$ denotes the Laplacian operator of $X_{t}$ applied to the restriction $h|_{X_t}$.
\\

As in \cite[Proposition 2.3]{AlexRadeschi}, one can prove the next lemma, taking in consideration item (c)
of Definition \ref{definition-flow}.

\begin{lemma}
\label{lemma-varhi-C1}
The restriction of the flow $\varphi$ to $\mathsf{N}/\FN\times (-\epsilon,\epsilon)$ is a $C^{1}$  map.
\end{lemma}

  From the lemma above, $\ddt u_h$ is continuous and we can define 
  $$f(\cdot,t):=-\phi(t)^*\left(\bigtriangleup h\right)+ \ddt\left(\phi(t)^*h\right)= -u_{\bigtriangleup h}+\ddt u_h(\cdot,t)$$ and thus we have
\begin{equation}
\label{eq-weak-parabolic-solution}
\ddt u_h-\bigtriangleup u_h=f\qquad \mathrm{in  }\; \X\times I
\end{equation}
in a weak sense. 

The goal is to use regularity properties of parabolic equations to prove that $u_h$ is smooth in $\X\times I$.
We first need to prove some initial regularity for $u_h$.

\begin{lemma}
\label{time-derivative-u-C1}
For $n\geq 0$ we have
\begin{enumerate}
\item[(a)] $\ddtn u_h(\cdot,t)\in C^{\infty}(\X)$ for each $t$.
\item[(b)] $\ddtn u_h\in L^{2}(0,T,H^{1}(\X)).$
\end{enumerate}
\end{lemma}
\begin{proof}
Let us prove the case where $n=1$; the other cases are identical. Consider the first blow-up $\hat{\rho}:\widehat{\newN}\to \newN$ of $(\newN,\FN)$ along its minimal stratum, cf. Section \ref{new-section-blow-up}. Since $\varphi$ is a flow of isometries on $(\newN,\tilde{\mathrm{g}})/\FN$, by Proposition \ref{new-lemma-lifted-flow-isometry} there is an induced flow of isometries $\hat\varphi$ on $\widehat{\newN}/\widehat{\FN}$. Since the number of strata in $(\widehat{\newN},\widehat{\FN})$ is strictly smaller than the number of strata in $(\newN,\FN)$, then by induction Theorem \ref{main-theorem} holds, and $\hat{\varphi}$ is smooth. Therefore the conditions of Proposition \ref{new-B-are-X-closed} are met, and $\varphi^{*}(\bar{Y}\cdot h)$ is a blow-up function as well.
Since
\begin{equation*}
\ddt u_h= \ddt\varphi^*h=\varphi^{*}(\bar{Y}\cdot h)
\end{equation*}
then by Proposition, \ref{new-lemma-tecnhical} $\ddt u_h(\cdot,t)$ is $C^{1}(\X)$ for each $t$.
The above equation also implies that $\ddt u_h$ is continuous. 

Note that in the regular stratum 
$$ \bigtriangleup \left(\ddt u_h\right)=\ddt\left(\bigtriangleup u_h\right)  \stackrel{\eqref{eq-weak-laplacianterm}}{=}\ddt u_{\bigtriangleup h} $$
Since   $\ddt u_h(\cdot,t)$ and  $\ddt u_{\bigtriangleup h}(\cdot,t) \in C^{1}(\X)$ we can apply the  same argument as in \cite{AlexRadeschi} to infer that the following equation holds weakly:
\begin{equation} 
\label{eq-laplacian-comutate-devivativetime}
 \bigtriangleup \left(\ddt u_h\right)=\ddt u_{\bigtriangleup h }\qquad\textrm{in } \X
\end{equation}
From regularity theory of solutions of  elliptic partial differential equations \cite{Evans} we conclude that $\ddt u_h(\cdot,t)$ lies in the Sobolev space $H^{3}(\X).$ Applying the argument successively, we obtain $\ddt u_h(\cdot,t)\in C^\infty(\X)$ and this concludes the proof of part (a).

Now since $\bar{Y}\cdot h\in \mathcal{B}$ and $\varphi_{t}$ preserves the geodesics  orthogonal to the minimal stratum, we have that
the directional derivatives of $\ddt u$ exist. Moreover
\begin{equation*}
\left\|\nabla \left(\ddt u_h(\cdot,t)\right)\right\|=\left\| \nabla \phi(t)^{*}(\bar{Y}\cdot h)\right\|.
\end{equation*}
The term in the right-hand side of this equation is  a continuous function due to Lemma \ref{lemma-varhi-C1}. 

Therefore, since  $\ddt u_h$ and $\left\|\nabla \left(\ddt u_h\right)\right\|$ are both continuous, part (b) follows. 
\end{proof}

Let us recall the next result  that can be found in Evans \cite[Theorem 6, page 365]{Evans}; 
see the same reference for definitions and notations about Sobolev Spaces.


\begin{theorem}[Regularity of parabolic equations]
\label{regularity-parabolic-equation}
Assume that $g\in H^{2m+1}(\X)$, $\frac{d^{k} f}{d t^{k}}\in L^{2}(0,T,H^{2m-2k}(\X))$ ($k=0,\ldots,m$). Suppose also the following $m^{th}$-order compatibility condition holds:
\[\left\{\begin{array}{l}
g_{0}:=g\in H_{0}^{1}(\X)\\
g_{1}:=f(0)+ \bigtriangleup g_{0}\in H_{0}^{1}(\X)\\
\ldots\\
g_{m}:=\frac{d^{m-1}f}{d t^{m-1}}(0)+\bigtriangleup g_{m-1}\in H_{0}^{1}(\X) \end{array}\right. \]
Let $u\in L^{2}(0,T,H_{0}^{1}(\X))$ with $\frac{d u}{d t}\in L^{2}(0,T, H^{-1}(\X))$ be a weak solution of 
\begin{eqnarray*}
\frac{d}{d t} u-\bigtriangleup u&=&f\qquad \mathrm{in } \; \X\times (0,T]\\
u&=&0\qquad \mathrm{on } \; \partial \X\times[0,T]\\
u&=&g\qquad  \mathrm{on } \; \X\times \{t=0 \}
\end{eqnarray*}
Then 
$\frac{d^{k}u}{d t^{k}}\in L^{2}(0,T,H^{2m+2-2k}(\X)).$
\end{theorem}

We can now finally prove the proposition below.

\begin{proposition}
\label{proposition-varhi-smooth-onfiber}
$u_h$ is smooth on $\X\times (-\epsilon,\epsilon)$.
\end{proposition}

\begin{proof}
Here let us make the following convenient definitions:
\begin{enumerate}\itemsep0.5em
\item $u_h^{(n)}:=\ddtn u_h$.
\item $f^{(n)}:=-u_{\bigtriangleup h}^{(n)}+u_h^{(n+1)}$.
\item $g^{(n)}:=u_h^{(n)}(\cdot,0)$.
\end{enumerate}

By differentiating equation \eqref{eq-weak-parabolic-solution}, one gets the following family of parabolic equations, parametrized by $n$:
\[
\left\{
\begin{array}{r l}
\frac{d}{d t} u_h^{(n)}-\bigtriangleup u_h^{(n)}=f^{(n)}&\mathrm{in } \; \X\times [0,T]\\
&\\
u^{(n)}_h=g^{(n)}& \mathrm{on } \; \X\times \{t=0 \}
\end{array}
\right.
\]

By Lemma \ref{time-derivative-u-C1}, $f^{(n)}(\cdot,0)$ and $g^{(n)}$ are $C^{\infty}(\X)$ for any $n>0$, and in particular all the $m^{th}$-order compatibility conditions hold. Therefore the only condition that has to be checked is
\begin{equation}\label{eq-last-condition}
f^{(n)}\in L^2(0,\epsilon, H^{2m}(\X)).
\end{equation}
If one first applies Theorem \ref{regularity-parabolic-equation} with $m=0$, then the condition \eqref{eq-last-condition} holds (recall Lemma \ref{time-derivative-u-C1}), and from the Theorem \ref{regularity-parabolic-equation} one gets $u^{(n)}_h\in L^2(0,\epsilon,H^2(\X))$.
\\

Now suppose by induction (on $r$) that $u_h^{(n)}\in L^2(0,\epsilon,H^{2r}(\X))$ for every $n$. Then $f^{(n)}\in L^2(0,\epsilon,H^{2r}(\X))$ as well, and one can apply Theorem \ref{regularity-parabolic-equation} with $m=r$. Again the only condition to be checked is \eqref{eq-last-condition}, which holds, and by the Regularity Theorem we obtain, in particular, that $u^{(n)}_h\in L^2(0,\epsilon,H^{2r+2}(\X))$.

By induction, we obtain that
\[
u_h,\,\ddtn u_h\in L^2(0,\epsilon,H^m(\X))\qquad \forall m\in \N
\]
and from this it follows that $u_h\in C^{\infty}(\X\timesÊ[0,\epsilon])$.
\end{proof}

\medskip

\subsection{The second step}
\begin{proposition}
\label{proposition-varhi-smooth-onsubbundle}
$\varphi:(\newN/\FN)\times I\to \newN/\FN$ is smooth.
\end{proposition}
\begin{proof}
If $Y$ is a point, then the result was already proved in the previous proposition. Let us assume that $Y$ is not a point.

We know from Proposition \ref{proposition-varhi-smooth-onfiber} that the restriction
$$\psi:=\varphi|_{\X/\FN\times I}:\X/\FN\times I\longrightarrow \newN/\FN$$
 is smooth, and therefore we can apply the inverse function theorem on orbit space (see \cite[page 45]{Schwarz}) to conclude that $\psi^{-1}$ is smooth.
 Note that,  for each fixed $s$  the function $\Projection_{N/\FN}\circ\psi(\cdot,s)$ is a constant $k(s)$. 
We claim that the diagram below commutes, and hence $\varphi$ is a composition of smooth maps and therefore  is a smooth map.
\begin{displaymath}
\xymatrix@+30pt{
\newN/\FN\times I \ar[r]^{\varphi} \ar[d]_{\psi^{-1}\times \mathrm{Id}} & \newN/\FN\ar[d]^{\psi^{-1}} \\
(\X/\FN\times I)\times I\ar[r]_{(\Projection_{1},\Projection_{2}+\Projection_{3})} & (\X/\FN\times I)
}
\end{displaymath} 
In fact, set $z =\psi(x^{*},s)$. Then we have
\begin{eqnarray*}
\varphi(z,t)&=& \varphi_{t}(z)\\
              &=& \varphi_{t}(\psi(x^{*},s))\\ 
              &=& \varphi(x^{*},s+t)\\
&=&\psi(\Projection_{1}\circ \psi^{-1}(z), \Projection_{2}\circ \psi^{-1}(z)+ t).
\end{eqnarray*}
This proves the commutativity of the diagram.
The smoothness of the arrows of the diagram can be proved using the smoothness of $\psi^{-1}$ and Schwarz's Lemma \cite{Schwarz-invariantfunctions}; see
also  comments in the beginning  of proof of the main Theorem \cite[page 65]{Schwarz-invariantfunctions}.
 
\end{proof}

\subsection{Proof of Corollary \ref{corollary}}
\label{section-corollary}
Given a single flow of isometries $\varphi: M/K\times I \to M/K$, around each point $q^*$ of $M/K$ we consider, as in the proof of Theorem 1.1, a neighbourhood $Y^*$ of $q^*$ in its $\varphi$-orbit. The preimage $\satsub=\pi_\F^{-1}(Y^*)$ is a regularly saturated submanifold of $M$, and Lemma \ref{lemma-thegroupG}  implies that $B_\e(\satsub) /\F=\newN/\FN=\newN/G$, where $G$ is compact. Therefore
 one can apply Schwarz's results \cite{Schwarz-invariantfunctions} on smooth maps on the orbit spaces and prove that the map $\newN/G\times I^2\to \newN/G\times I$, $(x^*,s,t)\mapsto (\varphi_s(x^*),t)$ is smooth. This implies that the function $M/K\times I^2\to I$, defined as before, is smooth as well. In particular, for any two smooth flows $\varphi,\psi:M/K\times I^2\to M/K$ the composition $M/K\times I^2 \to M/K$, $(x^{*},s,t)\mapsto \varphi_s(\psi_t(x^{*}))$ is smooth; see \cite[p.64]{Schwarz-invariantfunctions}. 

Consider now an isometric action $\mu:H\times (M/K)\to M/K$.  
Let $v_1,\ldots, v_h$, $h=\dim H$, be a basis of the Lie algebra of $H$. The above consideration and 
Theorem \ref{main-theorem} imply that the map
\begin{align*}
\psi:M/K\times I^{h}&\to M/K,\\
(x^*,t_{1},\ldots,t_{h})&\mapsto \mu(\exp(t_{1}v_{1})\cdots\exp(t_{h}v_{h}), x^{*})
\end{align*}
is smooth. This means that the action
is smooth on a neigborhood of $(e,x^{*})$ where $e$ is the identity of $H$ and $x^{*}$ is an arbitrary point of $M/K$.

In order to check that the action is smooth on a neighborhood of a generic point $(h,x^{*})$, it suffices to prove that the composition
\[
(t_1,\ldots,t_{n},x^{*})\mapsto  \psi(t_{1},\ldots,t_{n}, \mu(h, x^{*})) 
\]
 is smooth on a neigborhood of $(0,x^{*})$.

Note that since $H$ is connected, there exists a curve $\varphi_t$ of isometries such that $\varphi_t=h$ and $\varphi_0=e$. Therefore
\cite{AlexRadeschi} implies that the map  $x^{*}\mapsto \mu(h, x^{*})$ is smooth and the result follows from the above discussion.


\section{Molino's conjeture: proof of Theorem \ref{theorem-molinos-conjecture-orbitlike} }
\label{sec-proof-molino}

Let $(M,\F)$ be a singular Riemannian Foliation, and let $\overline{\F}$ be the partition of $M$ by the closures $\overline{L}$ of leaves $L\in\F$. In Molino \cite[Theorem 6.2, page 214]{Molino} (cf. \cite[Appendix D]{Molino} when $M$ non compact)  it is proved that each closure $\overline{L}$ is a closed submanifold, and that the partition $\overline{\F}=\{\overline{L}\}_{L\in\F}$ is a transnormal system, i.e., the leaves of $\overline{\F}$ are locally equidistant (cf. Definition \ref{definition-srf}). In fact, the equifocality of $\F$ (cf. \cite{AlexToeben2}) implies that plaques of $\F$ are equidistant to any fixed plaque of $\overline{L}_{q}$ and so are the plaques of $\overline{\F}$; see a similar argument in  \cite[Proposition 2.13]{Alex6}.
\\

In what follows we will prove that this partition is a smooth singular foliation.

\begin{proposition}
Let $\F$ be a orbit-like foliation. Then  
$(M,\overline{\F})$ is a singular  foliation.
\end{proposition}
\begin{proof}
Let us fix a point $q\in M$ and consider $\Sigma$ the stratum containing the point $q$. Consider a slice $D$ of the (regular) foliation $(\Sigma, \F_{\Sigma}).$
Finally let $v\in \nu_qL_q \cap T_{q}\overline{L_q}$. We want to prove that there exists a vector field $\vec{Y}$ around $q$, tangent to the leaves of $\overline{\F}$ so that $\vec{Y}(q)=v$. 


Let $(\newN,\FNo)$ be the local reduction of $(M,\F)$ along the slice $D$; recall Section \ref{new-subsection-local-reduction}. In what follows we will only make use of the metric $\mathrm{\tilde{g}}$ on $\newN$ introduced in Proposition \ref{new-lemma-same-transverse-metricN}, so we will give it as understood.
Since $\F$ is orbit-like,  $(\newN_q, \FNo)$ is homogeneous  given by the orbits of some group $G$; see Proposition \ref{new-corollary-Forbitlike-homogeneous}.

Recall that there is a pseudogroup $\H$ of local isometries of $\newN/\FNo$ that describes how the leaves around $\Sigma$ intersect $\newN$ (cf.  Section \ref{new-subsection-local-reduction}).

By applying Molino's theorem  to the regular foliation  $\F|_{\Sigma}$, ${\pi_{\FN}}_{*}v$ is tangent to the orbit of the closure of $\mathcal{H}$; see \cite[Theorem 5.1, page 156]{Molino} and  \cite[Section 3.4, page 287]{Molino}.

Let $(\widehat{\newN}, \widehat{\FNo})$ be the desingularization of $(\newN, \FNo)$ (cf. Remark \ref{remark-new-desingularization}), with projections 
$$
\hat{\rho}:\widehat{\newN}\to \newN \qquad\qquad \hat{\rho}_\#:\widehat{\newN}/\widehat{\FNo}\to \newN/\FNo.
$$
For the sake of simplicity, here we are using the notation $\widehat{\newN}$ to denote the desingularized space and not just the first blow-up along the minimal stratum. 
As in Proposition \ref{new-lemma-lifted-flow-isometry} 
we can lift any isometry  $\tau$ (recall equation  \eqref{new-definition-tau}) to an isometry $\hat{\tau}$ of the orbifold 
$(\widehat{\newN},\widehat{\tilde{\mathrm{g}}}) /\widehat{\FN}$; see also Remark \ref{remark-blowup-propositions-moregeneral} applied to small relatively compact neighborhoods. Let $\hat{\mathcal{H}}$ be pseudogroup generated by all isometries $\hat{\tau}$ constructed in this way, and let $\overline{\hat{\mathcal{H}}}$ its closure.  
By Salem \cite[Appendix D]{Molino} the closure $\overline{\hat{\mathcal{H}}}$ is a Lie pseudogroup. 
\\

We claim that we can find 
a Killing vector field $\vec{\hat{Y}}$ in the orbifold $\widehat{\newN}/\widehat{\FNo}$ with flow
 $\hat{\varphi}_{t}\in \overline{\hat{\mathcal{H}}}$, that projects to a flow $\varphi_t$ on $\newN/\FNo$ and such that $\ddt \varphi_t(q)={\pi_{\FNo}}_{*}v$.

Indeed, there is an et\'ale morphism
$$
\hat{\mathcal{H}} \to \mathcal{H}.
$$

Set $\widehat{D}:=\hat{\rho}^{-1}(D)$. Consider a point  $\hat{q}\in \widehat{\newN}/\widehat{\FNo}$ with $\hat{\rho}_\#(\hat{q})=q$. This point $\hat{q}$ can be chosen so that the
restriction of $\hat{\rho}_\#$ to a neighborhood of $\hat{q}$ coincides with the first blow-up. \textsc{}
We have that the restriction of $(\hat{\rho}_{\#})|_{\widehat{D}}$ to a neighborhood of the orbit $\overline{\hat{\mathcal{H}}(\hat{q})}$ is a submersion. Note that 
the restriction of $\hat{\rho}_\#$ to the orbit $\overline{\hat{\mathcal{H}}(\hat{q})}$ is a surjective smooth map
\[
\hat{\rho}_{\hat{q}}:\overline{\hat{\mathcal{H}}}(\hat{q})\to \overline{\mathcal{H}}(q).
\]
 Moreover, if the differential has rank $d$ at some $\hat{p}\in \overline{\hat{\mathcal{H}}}(\hat{q})$, it will be $d$ on the whole orbit $\hat{\mathcal{H}}(\hat{p})\In \overline{\hat{\mathcal{H}}}(\hat{q})$, and this implies that the differential of $\hat{\rho}_{\hat{q}}$ is everywhere constant, and hence $\hat{\rho}_{\hat{q}}$ is a submersion. In particular, given ${\pi_{\FNo}}_{*}v$ tangent to $\overline{\mathcal{H}}(q)$, there exists a vector $\hat{v}$ tangent to $\overline{\hat{\mathcal{H}}}(\hat{q})$ that projects to ${\pi_{\FNo}}_{*}v$. We can now take a flow $\hat{\varphi}\in \overline{\hat{\mathcal{H}}}$ with vector field $\vec{\hat{Y}}$ such that $\vec{\hat{Y}}(\hat{q})=\hat{v}$. Due to the construction of $\hat{\mathcal{H}}$, it makes sense to project $\hat{\varphi}$ to a flow $\varphi$ in $\newN/\FNo$, and this proves the claim.
\\

Since the foliation $(\newN,\FNo)$ is homogenous, it follows from Theorem \ref{main-theorem} that $\varphi$ is smooth.
From  Schwarz \cite[Corollary 2.4]{Schwarz} we can lift $\varphi$ and produce the smooth vector field $\vec{Y}$ on $\newN$ tangent to $v$. Finally, using the flow of vector fields tangent to the leaves, one can extend the vector field $\vec{Y}$ on $\newN$ to a vector field on an neighborhood of  $q$ in $M$ that is tangent to the leaves of $\F$ and this concludes the proof.

\end{proof}

\section{New tools}\label{new-tools}
In the following sections we define the main new technical tools used in this paper. These are of independent interest and we hope they might be used in other contexts, from which the decision of devoting a section to them is made.

The first tool is the blow-up of a singular Riemannian foliation. This object has already been studied in \cite{Alex6}, but here we will further the analysis and in particular we will define \emph{blow-up functions}. The second tool is the concept of \emph{local reduction}.

The reader mainly interested in the proof of Theorems \ref{main-theorem} and  
\ref{theorem-molinos-conjecture-orbitlike}  can skip the proofs presented in Sections \ref{section-blow-up}, \ref{sec-blow-upfunctions} and \ref{subsection-local-reduction}.

\subsection{Blow-up}
\label{section-blow-up}

Let $M$ be a complete manifold, and $(M,\F)$ a SRF with closed leaves. As in the classical theory of isometric actions, it is possible to construct, via compositions of projective blow-up's, a surjective map $\rho_{\epsilon} : M_{\epsilon}\to  M$ with the following properties:
\begin{enumerate}
\item $M_{\epsilon}$ is a smooth complete Riemannian manifold foliated by a regular Riemannian foliation $\F_{\epsilon}$.
\item The map  $\rho_{\epsilon}$ sends leaves of $\F_{\epsilon}$ to leaves of $\F$.
\end{enumerate}
This map is called a \emph{desingularization map}. If $M/\F$ is compact, then for each small $\epsilon>0$ one can choose $M_{\epsilon}$ and $\rho_{\epsilon}$ so that $d_{GH}(M_{\epsilon}/ \F_{\epsilon}, M/\F)<\epsilon$ where $d_{GH}$ is the Gromov-Hausdorff distance; see  \cite{Alex6}. 

In this section we briefly recall the construction of the first blow-up along the minimal stratum (see \cite{MolinoBlowup, Alex6,AlexBriquetToeben})   and present  Proposition \ref{lemma-lifted-flow-isometry},
the main result of this section.

Throughout this section, $\Sigma$ will denote  a relatively compact  neighborhood  in the minimal stratum of $M$. 
If the minimal stratum is compact, one can, equivalently, take $\Sigma$ to be the whole minimal stratum.

Following  a  procedure analogous to the blow-up of isometric actions one has the next lemma.

\begin{lemma}
\label{lemma-projectiveblow-usp}
Let $B:=\tub_{r}(\minimalstratum)$ be a small neighborhood of $\minimalstratum$. Then
\begin{enumerate}
\item[(a)] $\widehat{B}:=\{(x,[\xi])\in B \times \mathbb{P}(\nu \minimalstratum)| x= \exp^{\perp}(t\xi)\ \mbox{for}\ |t|<r\}$ is a smooth manifold (called blow-up  of $B$  along $\minimalstratum$) and the projection or blow-up map $\hat{\rho}:\widehat{B}\rightarrow B$, defined as $\hat{\rho}(x,[\xi])=x$ is also smooth.
\item[(b)]$\widehat{\minimalstratum}:=\hat{\rho}^{-1}(\minimalstratum)=\{(\hat{\pi}([\xi]),[\xi])\in \widehat{B}\}=\mathbb{P}(\nu\minimalstratum)$, where $\hat{\pi}: \mathbb{P}(\nu \minimalstratum)\rightarrow \minimalstratum$ is the canonical projection.
\item[(c)] There exists a  singular foliation $\widehat{\F}$ on $\widehat{B}$ so that  $\hat{\rho}: (\widehat{B}-\hat{\minimalstratum}, \hat{\F})\rightarrow (B-\minimalstratum, \F)$ is a foliated diffeomorphism. In addition if $\F$ is homogeneous then the leaves of $\widehat{\F}$ are also homogeneous. 
\end{enumerate}
\end{lemma}

Getting the right metric on $\widehat{B}$ is a bit more complicated.

\begin{lemma}[\cite{Alex6}]
There exists a metric $\hat{\mathrm{g}}$ on $\widehat{B}$ such that $\widehat{\F}$ is a SRF.
\end{lemma}
\begin{proof}
Let us briefly recall the construction of this metric, that will be important in the proof of Proposition \ref{lemma-lifted-flow-isometry}. 

Consider the smooth distribution $\Sl$ on $B$ defined as  
$\Sl_{\exp(\xi)}:=T_{\exp(\xi)}S_{q}$ where $\xi\in \nu_{q}\minimalstratum$ and $S_{q}$ is a slice of $L_{q}$  at $q$  with respect to the original metric $g$.

First we find a metric $\tilde{g}$ with the following properties:

\begin{enumerate}
\item[(a)] The distance between the leaves of $\F$ on $B$ with respect to $\tilde{\mathrm{g}}$ and to respect to $\mathrm{g}$ are the same.
\item[(b)] The normal space  of each plaque of $\F|_{B}$ (with respect to $\tilde{\mathrm{g}}$) is contained in $\Sl.$ In fact those spaces are the orthogonal projection (with respect to $\mathrm{g}$) of the normal spaces (with respect to $\mathrm{g}$) of $\F|_{B}.$
\item[(c)] If a curve $\gamma$ is a unit speed geodesic segment orthogonal to $\minimalstratum$ with respect to the original metric $\mathrm{g}$, then
$\gamma$ is a unit speed geodesic segment orthogonal to $\minimalstratum$ with respect to the new metric $\tilde{\mathrm{g}}$.
\end{enumerate}

We now come to the second step of our construction, in which we change  the metric $\tilde{\mathrm{g}}$ in
some directions, getting a new metric $\hat{\mathrm{g}}^{B}$ on $B-\minimalstratum$.

First note that, for small $\xi\in \nu_{q}\minimalstratum,$ we can decompose $T_{\exp_{q}(\xi)}M$ as a direct sum of orthogonal subspaces (with respect to the metric $\tilde{g}$) 
\begin{equation}
\label{eq-decomposition-H}
T_{\exp_{q}(\xi)}M= \Sl_{\exp_{q}(\xi)}^{\perp}\oplus \Sl_{\exp_{q}(\xi)}^{1}\oplus \Sl_{\exp_{q}(\xi)}^{2}\oplus \Sl_{\exp_{q}(\xi)}^{3},
\end{equation}
where $\Sl_{\exp_{q}(\xi)}^{\perp}$ is orthogonal to $\Sl_{\exp_{q}(\xi)}$ and  $\Sl_{\exp_{q}(\xi)}^{i}\subset \Sl_{\exp_{q}(\xi)}$, i=1,2,3, are defined below:
\begin{enumerate}
\item $\Sl_{\exp_{q}(\xi)}^{1}$ is the tangent space of the normal sphere $\exp(\nu_q\Sigma)\cap \partial B_{	\|\xi\|}(q)$,
\item $\Sl_{\exp_{q}(\xi)}^{2}$ is the line generated by $\frac{d}{d t}\exp_{q}(t\xi)|_{t=1}$.
\item $\Sl_{\exp_{q}(\xi)}^{3}$ is the orthogonal complement of $\Sl_{\exp_{q}(\xi)}^{1}\oplus \Sl_{\exp_{q}(\xi)}^{2}$ in $\Sl_{\exp_{q}(\xi)}$.
\end{enumerate}

Now we define a new metric $\hat{\mathrm{g}}^{B}$ on
$B-\minimalstratum$ as follows:

\begin{equation}
\label{def-metrica-hatgM}
\hat{\mathrm{g}}^{B}_{\exp_{q}(\xi)}(Z,W):= \tilde{\mathrm{g}}(Z_{\perp},W_{\perp})+ \frac{r^{2}}{\|\xi\|^{2}} \tilde{\mathrm{g}}(Z_{1},W_{1})
                                +\tilde{\mathrm{g}}(Z_{2},W_{2})+  \tilde{g}(Z_{3},W_{3}),
\end{equation}
where $Z_{i},W_{i}\in \Sl_{\exp_{q}(\xi)}^{i}$ and  $Z_{\perp},W_{\perp}\in \Sl_{\exp_{q}(\xi)}^{\perp}.$

Finally we define the pullback metric $\hat{\mathrm{g}}:=\hat{\rho}^{*}\hat{\mathrm{g}}^{B}$.
\end{proof}

We have recalled the construction of the blow-up an $\F$-invariant neigbhoorhood $B$ along $\Sigma$. We have explained the case where $B=\tub_r(\Sigma)$ because we will only be  concerned with this kind of neighborhood $B$ and with this first blow-up $\hat{\rho}$.

\begin{remark}
\label{remark-glue}
For the sake of completeness let us explain the rest of the construction, e.g. when $M$ is compact.
 We  simply glue $\widehat{B}$ with a copy of $M-B$ and  construct the space
$\widehat{M}_{r}(\minimalstratum)$ and  the projection $\rho_{r}:\widehat{M}_{r}(\minimalstratum)\rightarrow M$.
A natural singular foliation  $\widehat{\F}_{r}$ is induced on $\widehat{M}_{r}(\minimalstratum)$ in analogy to the blow-up of isometric actions. 
To define the appropriate metric   $\widehat{\mathrm{g}}_{r}$ on  $\widehat{M}_{r}(\minimalstratum)$ consider  a partition of unity of $\widehat{M}_{r}(\minimalstratum)$ by two functions $\hat{f}$ and $\hat{h}$ such that
 \begin{enumerate}
 \item $\hat{f}=1$ in $\tub_{r/2}(\widehat{\minimalstratum})$  and $\hat{f}=0$ outside of $B$.
 \item $\hat{f}$ and $\hat{h}$ are constant on the cylinders  $\partial \tub_{\delta}(\widehat{\minimalstratum})$ for $ \delta<2 r.$
 \end{enumerate}
 Set $\hat{\mathrm{g}}_{r}:=\hat{f}\hat{\mathrm{g}}+\hat{h} \mathrm{g}_1$, where $\mathrm{g}_1$ is a metric that near to $\widehat{\minimalstratum}$ is $\tilde{\mathrm{g}}$
and far away is the original metric $\mathrm{g}$ so that $\F$ is a SRF with respect to $\mathrm{g}_1$. 
 The desingularization $\rho_{\epsilon}$ mentioned in the   beginning of this section  is then the composition of the blow-up's along the strata. 
\end{remark}

\begin{proposition}
\label{lemma-lifted-flow-isometry}
Each flow of isometries $\varphi:B/\F\times I\to B/\F$ can be lifted to a  flow of isometries 
$\hat{\varphi}: \widehat{B} /\widehat{\F} \times I \to \widehat{B} /\widehat{\F}.$ 
\end{proposition}
\begin{proof}

Since $\varphi_t$ maps geodesics orthogonal to the minimal stratum to geodesics orthogonal to the minimal stratum, the lift $\hat{\varphi}_t$ is well defined and continuous.

 Let $x$ be a principal point and $H$ be the transverse space of the leaf $L_x.$ Then $H$ is decomposed into a direct sum of subspaces $H_1\oplus H_{2}\oplus H_3$, 
 where $H_{i}=\Sl_{i}\cap H$; for the definition of $\Sl_{i}$ recall equation \eqref{eq-decomposition-H}.
Let $\widehat{H}$ be the transversal space of $\widehat{L}_{\hat{x}}$ where $\hat{\rho}(\hat{x})=x$. Then $\widehat{H}$ also decomposes into a direct sum $\widehat{H}_j$ and
$d \hat{\rho}: (\widehat{H}_{j},\hat{\mathrm{g}}_{T})\to (H_{j},g_j)$ is an isometry where $\hat{\mathrm{g}}_{T}$ is the transverse metric of $\widehat{\F}$ and $\mathrm{g}_{j}$ is the restriction of transvere metric $\mathrm{g}_{T}$ of $\mathcal{F}$ to $H_{j}$, if $j\neq 1$ and $\mathrm{g}_{1}=\frac{r^{2}}{\|\xi\|^{2}}\mathrm{g}_{T}.$

Note that $\varphi_t$ (respectively $\hat{\varphi}_{t}$) preserves the decomposition $H_i$ (respectively $\widehat{H}_{i}$). Since the function $\frac{r^{2}}{\|\xi\|^{2}}$ is invariant under the action of $\varphi_t$ we infer that  $\hat{\varphi}_{t}$ is a local isometry on $(\hat{\rho})^{-1}(B_{0})/\widehat{\F}$, where $B_0$ is the union of principal leaves of $B$. Using the density of principal points in the quotient space $\widehat{B}/\widehat{\F}$ and the fact that a minimal geodesic segment joining principal points does not contain singular points, we conclude that the each $\hat{\varphi}_t$ is a global isometry on $\widehat{B}/\widehat{\F}$.

\end{proof}

\begin{remark}
\label{remark-locally bounded}
Using blow-ups, one can also check that
the derivative $\bar{Y}$ of the flow $\varphi$ is locally bounded (cf. Definition \ref{definition-flow}).
In fact, by successive blow-ups one can lift a continuous one parameter local group  $\varphi$ 
on $B/\F$
to an isometric flow on an orbifold (see Remark \ref{remark-blowup-propositions-moregeneral} below) where they are locally
bounded by more classical results, see e.g. \cite[Salem appendix D]{Molino}. Since the blow-up's  are distance non-increasing maps between the leaf spaces (see \cite[Remark 3.8]{Alex6}) the result follows.
\end{remark}

Although along the paper we will consider foliations $\F$ whose leaves are homogeneous but not necessarly Riemannian homogeneous, we present the next result for the sake of completeness. 

\begin{proposition}
\label{proposition-desingularization-homogeneous-foliation}
 Let $(B,G)$ be a Riemannian homogeneous SRF. Then $G$ acts on $\widehat{B}$, and there exists a new metric $\hat{\mathrm{g}}^{G}$ such that $G$ acts by isometries and $(\widehat{B},\widehat{\F})$ is the Riemannian homogeneous foliation induced by $G$. The transverse metric of $\hat{\mathrm{g}}^{G}$ coincides with the transverse metric of $\hat{\mathrm{g}}$. 
\end{proposition}
According to this proposition, in particular,  a  flow of isometries $\varphi$ on the orbit space $B/\F$  can be lifted to a  flow  of isometries $\hat{\varphi}$  on the orbit space $\widehat{B}/\widehat{\F}$  with respect to the new metric $\hat{\mathrm{g}}^{G}$.  
\begin{proof}
We first claim that the action of $G$ on each stratum preserves the normal bundle (with respect to $\tilde{\mathrm{g}}$) of each orbit in this stratum. In addition $G$ acts isometrically on the fibers of this bundle.

The above claim is a direct consequence of the following facts:
\begin{enumerate}
\item The distribution $\Sl$ is invariant under the action of $G$.
\item The normal bundle of the orbits (with respect to the original metric $\mathrm{g}$) is invariant under the action of $G$.
\item The orthogonal projection (with respect to the original metric $\mathrm{g}$) is also invariant under the action of $G$.
\end{enumerate} 

Now, since the action preserves the decomposition $H_1$, $H_2$ and $H_3$ the claim is also valid for the metric $\hat{\mathrm{g}}^{B}$ and hence to the blow-up metric $\hat{\mathrm{g}}=\hat{\rho}^{*}\hat{\mathrm{g}}^{B}$. Finally one can define the new metric as 
$$\hat{\mathrm{g}}^{G}(X,Y):=\int_{G} \hat{\mathrm{g}}(dg X, dg Y)\omega $$ 
where $\omega$ is a right-invariant volume form of the compact group $G$.
The rest of the proof follows from Proposition \ref{lemma-lifted-flow-isometry}.
\end{proof}

\begin{remark}
\label{remark-blowup-propositions-moregeneral}
When $M$  and $\F$ are compact, we can take into account Remark \ref{remark-glue} and generalize the above results  as follows: 
\begin{enumerate}
\item[(a)] Each flow of isometries $\varphi:M/\F\times I\to M/\F$ can be lifted to a (smooth) flow of isometries 
$\varphi_{\epsilon}:M_{\epsilon}/\mathcal{F}_{\epsilon}\times I \to M_{\epsilon}/\mathcal{F}_{\epsilon}.$
\item[(b)] If $\F$ is the partition of $M$ into orbits of a compact group $G$ of isometries of $M$, then there exists a metric $\mathrm{g}_{\epsilon}^{G}$ on $M_{\epsilon}$ so that $\F_{\epsilon}$ turns out to be the partition of $M$ into orbits of an isometric action of $G$ on $M_{\epsilon}$. In addition the transverse metric associated to this new metric 
$\mathrm{g}_{\epsilon}^{G}$ coincides with the transverse metric of the original metric $\mathrm{g}_{\epsilon}$ of $M_{\epsilon}$. 
\end{enumerate}

\end{remark}

We conclude this section discussing the Lie structure of isometries groups of leaf spaces. The next proposition  is not necessary for the proof of the main results. 
\begin{proposition}
 \label{liegroupstructure}
Let $\F$ be a closed SRF on a compact Riemannian manifold $M$. 
Then each connected compact group $H$  of isometries on $M/\F$ is a Lie group.
\end{proposition}

\begin{proof}
 
Following \cite[section 5.1]{GorodskiLytchak} one can check that each  isometry of $H$  sends strata to strata. It also sends geodesics orthogonal to strata to geodesics orthogonal to strata.  Therefore, as explained  in Remark \ref{remark-blowup-propositions-moregeneral}, one can lift each isometry  $h\in H$
 to an isometry $\hat{h}$ on the orbifold $M_{\epsilon}/\F_{\epsilon}$. Let $\widehat{H}$ denote the group generated by these isometries.
Note that it can be identified with $H$. 

We claim that  $\widehat{H}$ is also compact. In fact, if $\{ \hat{h}_n \}$ is a sequence of $\widehat{H}$, by construction it projects to a sequence $\{ h_{n} \}$
on $H$. Since $H$ is compact, there is a subsequence $\{h_{n_i} \}$ that converges to an isometry $h\in H$. In particular its restriction to the boundary of a tube around a
 minimal stratum also converges in the compact open topology. Using this fact and the construction of the desingularization (that is a composition of blow-ups) 
one can conclude that the subsequence  $\{\hat{h}_{n_i}\}$ converges  to $\hat{h}$ (the lift of $h$) in the compact open topology  and hence $\widehat{H}$ is compact.

Since $\widehat{H}$ is a compact group of isometries on an orbifold, it  follows from  \cite[Theorem 2]{Bagaev-Zhukova} that $\widehat{H}$ and $H$ are a Lie group. 

\end{proof}

\subsection{Blow-up functions}
\label{sec-blow-upfunctions}

We now introduce a class of basic functions on $B$ that will be used  in Lemma \ref{time-derivative-u-C1} to
 check some regularity conditions necessary to prove the smoothness of solutions of a (weak) parabolic equation.
In particular the main results of this section are  Proposition \ref{lemma-tecnhical} and Proposition \ref{B-are-X-closed}.

Consider the blow-up $\hat{\rho}:(\widehat{B},\widehat{\F})\to (B,\F)$  of $B$ along its minimal stratum $\Sigma$.

\begin{definition}
\label{def-blow-upfunction}
We say that a continuous $\F$-basic function $h$ on $B$ belongs to $\mathcal{B}$ or is a \emph{blow-up function}, if
\begin{enumerate}
\item[(a)] $h\circ\hat{\rho}$ is a smooth $\widehat{\F}$-basic function on $\widehat{B}$.
\item[(b)]  The restriction of $h$ to $\Sigma$ and to $B-\Sigma$ is smooth.
\item[(c)] $X\cdot h=0$ for each $X\in \nu\Sigma$.
\end{enumerate} 
\end{definition}

In what follows we prove two important properties of these functions.
\begin{proposition}
\label{lemma-tecnhical}
If $h\in \mathcal{B}$ then  $h$ is a $C^{1}$ function. 
\end{proposition}
\begin{proof}
Let $q\in \Sigma$.
We claim that it suffices to show that
\begin{equation}
\label{lemma-tecnhical-eq0}
v_{n}\cdot h(p_{n})\to 0
\end{equation}
for every sequence $(p_{n},v_n)\to (q, v_0)$ such that $(p_n,v_n)\in TB$, and the vectors $v_n$ are tangent to the distance spheres in the normal bundle of $\Sigma$.

Consider first a simple example, namely when $B=\mathbb{R}^{2}$ is foliated by concentric circles around the origin and $\Sigma=\{q\}=\{(0,0)\}$. Note that $\frac{\partial}{\partial x}\big|_{p_{n}}= a_1\frac{\partial}{\partial r}\big|_{p_{n}} +a_{2}v_n$ and 
$\frac{\partial}{\partial y}|_{p_{n}}= b_1\frac{\partial}{\partial r}\big|_{p_{n}} +b_{2}v_{n}$ where $|a_{i}|$ and $|b_{i}|$ are bounded. Equation  \eqref{lemma-tecnhical-eq0} and item (c) of Definition \ref{def-blow-upfunction} imply that $\nabla h(p_{n})\to 0$, which is what one needs to ensure that $h$ is $C^1$.

In the general case, consider a converging sequence $(p_n,v_n)\to (q,v)$.  The tangent spaces $T_{p_n}B$ split as $\ker(\Projection_*)\oplus\ker(\Projection_*)^{\perp}$, where $\Projection:B\to \Sigma$ is the metric projection. Every $v_n$ then splits as $v_n=v_n^{\ker\,\Projection_*}+v_n^{\ker\,\Projection_*^{\perp}}$, and in order to prove the proposition we need to prove
\begin{eqnarray}
v_n^{\ker\,\Projection_*}\cdot h(p_n)&\stackrel{n\to \infty}{\longrightarrow}& v^{\ker\,\Projection_*}\cdot h(q) \label{lemma-technical-eqU}\\
v_n^{\ker\,\Projection_*^{\perp}}\cdot h(p_n)&\stackrel{n\to \infty}{\longrightarrow}& v^{\ker\,\Projection_*^{\perp}}\cdot h(q) \label{lemma-technical-eqW}
\end{eqnarray}
By \cite[proof of Lemma 3.5]{Alex6}, \eqref{lemma-technical-eqW} is satisfied. Moreover, since $h$ is a blow-up function, by condition (c) we have $v^{\ker\,\Projection_*}\cdot h(q)=0$, and it is enough to check that $v_n^{\ker\,\Projection_*}\cdot h(p_n)\to 0$. Since moreover the derivatives in the radial direction go to zero, this reduces to check \eqref{lemma-tecnhical-eq0}, as claimed.
\\

Set $g:=h\circ\hat{\rho}$  and let $\{\hat{p}_{n}\} \subset \widehat{B}$ be the lift of $\{p_{n}\}$. We can assume without loss of generality that $\{\hat{p}_{n}\}$ is contained in a  relatively compact neighborhood in $\widehat{B}$ that admits cylindrical coordinates $(r,\theta_i,z_i)$.

For each $p_n$, define a frame $\{\vec{e}_{j}\}$ of $(\ker\Projection_*)_{p_n}\cap \frac{\partial}{ \partial r}\big|_{p_n}^{\perp}$, where $\vec{e}_j(p_n)= \frac{1}{r} \hat{\rho}_* \frac{\partial}{\partial \theta_{j}}$. Also, let $\underline{\theta}=(\theta_1,\ldots \theta_{\codim(\Sigma,B)-1})$ and $\underline{z}=(z_1,\ldots z_{\dim \Sigma})$.

By definition of $\vec{e}_{j}$, in order to prove equation \eqref{lemma-tecnhical-eq0} it suffices to show 

\begin{equation}
\label{lemma-tecnhical-eq1}
\lim_{n\to \infty} \vec{e}_j(p_n)\cdot h(p_n)=\lim_{n\to\infty}  \frac{1}{r(n)} \frac{\partial g}{\partial \theta_i}(r(n),\underline{\theta}(n),\underline{z}(n))=0
\end{equation}
for the sequence $\hat{p}_{n}=(r(n),\underline{\theta}(n),\underline{z}(n))$ that converges to the fiber $(0,\underline{\theta},\underline{z})$ . 

From the definition of $g$ we have
\begin{equation}
\label{lemma-tecnhical-eq2}
\frac{\partial g}{\partial\theta_i}(0,\underline{\theta},\underline{z})=0
\end{equation}

Condition (c) implies
\begin{equation}
\label{lemma-tecnhical-eq3}
\frac{\partial g}{\partial r}(0,\underline{\theta},\underline{z})=0
\end{equation}
and hence, since $g$ is smooth (see condition (a)), we conclude that
\begin{equation}
\label{lemma-tecnhical-eq4}
\frac{\partial^{2} g}{\partial r \partial \theta_{i}}(0,\underline{\theta},\underline{z})=\frac{\partial^{2} g}{\partial \theta_{i}\partial r}(0,\underline{\theta},\underline{z})=0
\end{equation}

From mean value theorem we also have

\begin{equation}
\label{lemma-tecnhical-eq5}
\Big| \frac{1}{r(n)} \frac{\partial g}{\partial \theta_i}(r(n),\underline{\theta}(n),\underline{z}(n))- \frac{1}{r(n)}\frac{\partial g}{\partial \theta_i}(0,\underline{\theta}(n),\underline{z}(n))\Big|\leq \Big|\frac{\partial^{2} g}{\partial r \partial \theta_{i}}(\tilde{r}(n),\underline{\theta}(n),\underline{z}(n))\Big|
\end{equation}

Now Eq. \eqref{lemma-tecnhical-eq1} follows direct from equations \eqref{lemma-tecnhical-eq2},  \eqref{lemma-tecnhical-eq4} and \eqref{lemma-tecnhical-eq5}.
\end{proof}

Let $\varphi$ be a flow of isometries on $B/\F$ and 
consider the flow of isometries $\hat{\varphi}$ on  $\widehat{B}/\widehat{\F}$ defined in Proposition \ref{lemma-lifted-flow-isometry}. Let $\bar{Y}$ be the associated derivative in the quotient $B/\F$; recall Definition \ref{definition-flow}. 

\begin{proposition}
\label{B-are-X-closed}
Assume that $\hat{\varphi}$ is smooth. 
Then for each $h\in \mathcal{B}$ we have  $\varphi_{s}^{*} h$ and $\varphi_{s}^{*}(\bar{Y}\cdot h) \in \mathcal{B}$.
\end{proposition}
\begin{proof}
We must check that the function $(\bar{Y}\cdot h)\circ \varphi_{s}$ satisfies the conditions of the Definition \ref{def-blow-upfunction}.
The case  $\varphi_{s}^{*} h\in \mathcal{B}$ is simpler.

Condition (b) of Definition \ref{def-blow-upfunction} follows from  hypothesis.

Now we want to check condition (a) of Definition \ref{def-blow-upfunction}. 
Note that $ \varphi_{z+s}\circ\hat{\rho}=\hat{\rho}\circ\hat{\varphi}_{z+s}.$
\begin{eqnarray*}
(\bar{Y}\cdot h)\circ \varphi_{s}\circ\hat{\rho}(\cdot)&=&\frac{d}{d z} h\Big(\varphi\big(\varphi(\hat{\rho}(\cdot),s),z\big)\Big)\Big|_{z=0}\\
&=&\frac{d}{d z} h\Big(\varphi\big(\hat{\rho}(\cdot),s+z\big)\Big)\Big|_{z=0}\\
&=& \frac{d}{d z} h\Big(\hat{\rho}\big(\hat{\varphi}_{s+z}(\cdot)\big)\Big)\Big|_{z=0}\in C^{\infty}.
\end{eqnarray*}

Finally we have to check condition (c) of Definition \ref{def-blow-upfunction}. Let $\gamma$ be a geodesic orthogonal to the minimal stratum $\Sigma$ and $\hat{\gamma}\subset \widehat{B}$ a lift of $\gamma$. Consider the smooth function $g(z,t):=h\circ\hat{\rho}(\hat{\varphi}_{s+z}(\hat{\gamma}(t)))$.
Note that $\hat{\varphi}_{z+s}\circ\hat{\gamma}$ is a horizontal geodesic orthogonal to the lift of $\Sigma$ and hence that $\hat{\rho}\circ\hat{\varphi}_{z+s}\circ\hat{\gamma}$ is orthogonal to $\Sigma$. This fact and the fact that  $h\in \mathcal{B}$ (in particular satisfies condition (c) of Definition \ref{def-blow-upfunction}) imply that
 $\frac{\partial}{\partial t}g(z,0)=0.$ We  conclude that

\begin{eqnarray*}
\frac{d}{d t}(\bar{Y}\cdot h)\circ \varphi_{s}\circ\gamma(t))|_{t=0} &=& \frac{\partial^{2}}{\partial t \partial z} h\Big(\varphi_{s+z}\big(\gamma(t)\big)\Big)\Big|_{z,t=0}\\
 &=& \frac{\partial^{2}}{\partial t \partial z} h\Big(\hat{\rho}\circ\hat{\varphi}_{s+z}\big(\hat{\gamma}(t)\big)\Big)\Big|_{z,t=0}\\
 &=&\frac{\partial^{2}}{\partial t\partial z}g(z,t)|_{z,t=0}\\
 &=&\frac{\partial^{2}}{\partial z\partial t}g(z,t)|_{z,t=0}\\
 &=&0.
\end{eqnarray*}

\end{proof}

\begin{remark}
As we prove Theorem \ref{main-theorem}, it will be clear that the hypothesis in Proposition \ref{B-are-X-closed}, i.e., the smoothness of  $\hat{\varphi}$, is always satisfied when $\F$ is  homogeneous. 
\end{remark}

\begin{remark}
The above results are also valid for foliation with disconnected leaves.
\end{remark}


\subsection{The local reduction}
\label{subsection-local-reduction}

For the sake of completeness, we start by recalling the definitions presented in Section \ref{new-subsection-local-reduction}.
Let $(M,\F)$ be a SRF, and let $\Sigma\In M$ be a stratum of $\F$. 
Let $Y$ be a submanifold contained in a slice (a transverse submanifold) of the regular foliation $(\Sigma,\F|_{\Sigma})$. Consider 
$\satsub:=\pi_{\F}^{-1}(\pi_{\F}(Y))$ the \emph{saturation} of $Y$. We also assume that  $Y$  coincides with the
intersection of $\satsub$ with the slice, i.e., that $Y$ is invariant under the action of the holonomy
pseudogroup of the (regular) foliation $(\Sigma,\F_\Sigma)$. 
Suppose that the normal exponential $\exp:\nu(\satsub) \to M$ is well defined on a tubular neighborhood of radius $\e$ around $\satsub$, and call $B_{\e}\satsub$ the image of such tube. $B_{\e}\satsub$ exists if for example  $\satsub$ is relatively compact. 
Define $\newN:= \exp\left(\nu(\satsub)\big|_Y\right)\cap B_\e\satsub$, together with the metric projection $\Projection_Y:\newN\to Y$. The fiber of $\Projection_Y$ at a point $p\in Y$ will be denoted by $\newN_p$.  The submanifold $\newN$ is called \emph{local reduction of $(M,\F)$ along $Y$}.
The foliation $\F$ intersects $\newN$ in a foliation $\FNo:=\F \cap \newN$. 
Notice that the leaves of $\FNo$ are contained in the fibers of $\Projection_Y$.

\begin{proposition}
\label{lemma-same-transverse-metricN}
There exists a metric $\tilde{\mathrm{g}}$ on $\newN$ that preserves the transverse metric of $\F$, i.e., 
the distance between  leaves of $\FNo$ is the
 same as the distance between the plaques of $\F$ that contain such leaves.
In particular $\FNo$ is a  SRF on $(\newN,\tilde{\mathrm{g}})$.
\end{proposition}
\begin{proof}
This metric can be constructed as follows. Consider the regular distribution $\Sl$ defined as $\Sl_{z}:=T_{z} S_{p}$ where $z\in \newN_{p}$ and $S_p$ is the slice through $p$. According to \cite[Proposition 3.1]{Alex6} there exists a metric $\tilde{\mathrm{g}}$ on a neighborhood of $\newN$ 
so that the normal space of $\F$ (with respect to $\tilde{\mathrm{g}}$) is contained in $\Sl$ and the SRF $\F$ (with respect to $\tilde{\mathrm{g}}$) has the same transverse metric of $\F$ (with respect to the original metric).
Let $\Pi: TM|_\newN\to T\newN$ be the orthogonal projection (with respect to original metric) and define a metric on $T\newN$ as $(\Pi|_{\mathcal{H}}^{-1})^{*} \tilde{\mathrm{g}}.$ Let us denote this new metric on $\newN$ also as $\tilde{\mathrm{g}}.$
Following \cite[Proposition 2.17]{Alex6} we conclude that $\FNo$ is a SRF on $(\newN,\tilde{\mathrm{g}})$. 
\end{proof}

Suppose $\newN$ is the local reduction of $(M,\F)$ along $Y$, and let $q\in Y$. If $q'$ is another point in $L_q$, we can similarly find $Y'$ through $q'$ and a local reduction $(\newN', \FNNo)$ along $Y'$. Moreover we can do it so that there is a flow of a vector field $X$ tangent to the leaves that sends $(\newN,\FNo)$ foliated diffeomorphically to $(\newN',\FNNo)$. By the properties of the metrics $\tilde{\mathrm{g}}$, $\tilde{\mathrm{g}}'$ on $\newN$, $\newN'$ proved in Proposition \ref{lemma-same-transverse-metricN}, this diffeomorphism induces a local isometry
$$\tau:(\newN,\tilde{\mathrm{g}})/\FNo\longrightarrow (\newN',\tilde{\mathrm{g}}')/\FNNo.$$
This isometry does not depend on the choice of $X$, but only on the homotopy class of the integral curve $\alpha$ of $X$ joining $q$ to $q'$, so we refer to $\tau$ as $\tau_{[\alpha]} $.

Notice that $Y$ can meet $L_q$ in several points $q_i$. For every such $q_i$, and every curve $\alpha$ from $q$ to $q_i$ contained in a leaf, there is an associated local isometry

$$\tau_{[\alpha]}:\newN/\FNo\longrightarrow \newN/\FNo.$$

Let  $\mathcal{H}$ be the pseudogroup of $\newN/\FNo$ generated by all isometries $\tau_{[\alpha]}$, for all curves
$\alpha\In L_q$ with initial and final point in $\newN$.

\begin{definition}\label{singular-holonomy-pseudogroup}
The pseudogroup $\H$ defined above will be called \emph{singular holonomy pseudogroup}. The triple $(\newN,\FNo,\H)$ is an example of singular Riemannian foliation with disconnected leaves (cf. Section \ref{subsection-disconnected-leaves}) which we denote by $\FN$. The leaves of $\FN$ are precisely the (possibly disconnected) intersections of $\newN$ with the leaves of $\F$.
\end{definition}
\begin{remark}\label{remark-reduction}
By Proposition \ref{lemma-same-transverse-metricN}, the inclusion $\newN\to B_\e\satsub$ induces an isometry $\newN/\FN\to B_\e\satsub/\F$ that preserves the codimension of the leaves. By the main result in \cite{AlexRadeschi}, this map is smooth and in particular every smooth basic function in $(\newN,\FN)$ extends to a smooth basic function in $(B_\e\satsub,\F)$.
\end{remark}

Let $\pi_{\FNo}:\newN \to \newN /\FNo$ be the quotient map, $Y^*:=\pi_{\FNo}(Y)$ and $\Projection_{Y^*}: \newN /\FNo\to Y^*$ be the submetry with fibers $\newN_p/\FNo$. Note that $Y^*$ can be identified with $Y$. It is easy to see that 
\begin{equation}
\label{eq-comutativity-submersions}
\Projection_{Y}=\Projection_{Y^*}\circ\pi_{\FNo} 
\end{equation}
or, equivalently, that the following diagram commutes
\begin{displaymath}
\xymatrix@+20pt{
\newN \ar[d]_{\pi_{\FNo}}\ar[r]^{\Projection_{Y}} & Y=Y^* \\
\newN /\FNo \ar[ru]_{\Projection_{Y^*}}
}
\end{displaymath}

A local reduction satisfies the following nice property that relates the foliated structure of $(\newN,\FNo)$ with the submersion $\Projection_Y:\newN\to Y^*$.
\begin{proposition}
\label{lemma1-proposition-vectorfieldY-on-N}
Any horizontal basic vector field $\vec{\xi}$ for $\Projection_Y$ is a horizontal foliated vector field of $\FNo$ and for each fixed $q$ the geodesic $t\to \exp_{q}(t\vec{\xi}(q))$ is always contained in the same stratum. 
\end{proposition}

\begin{proof} 
Let $\vec{\xi}$ be a horizontal basic vector field of $\Projection_Y$. We first claim that $\vec{\xi}$ restricted to the regular stratum of $\FNo$ is a \emph{foliated vector field}, i.e., a basic vector field with respect to $\pi_{\FNo}$. For $q\in \newN_p$ a regular point, let $\tilde{\xi}$ be the horizontal foliated vector field along $L_{q}$ such that $\tilde{\xi}(q)=\vec{\xi}(q)$. Since $\Projection_Y$ is foliated and the $\FNo$-leaves in $Y$ are just points, the fibers of $\Projection_Y$ are saturated by the leaves of $\FNo$ and therefore $\tilde{\xi}$ is everywhere normal to $\newN_{p}$.

Let $q'\in L_q$. By equation \eqref{eq-comutativity-submersions}  we note that the $\Projection_Y$-horizontal geodesics $\alpha_q(t)=\exp_{q}\big(t\tilde{\xi}(q)\big)$, $\alpha_{q'}(t)=\exp_{q'}\big(t\tilde{\xi}(q')\big)$ project to the same geodesic in $Y$. This implies that $\tilde{\xi}(q')=\vec{\xi}(q')$ for every $q'\in L_{q}$ and concludes the proof of the claim.

By the equifocality property of singular Riemannian foliations (cf. \cite{AlexToeben2}), $\alpha_q$ cannot contain a singular point $\alpha_{q}(t_{0})$. In fact if such a point existed then there would be two lifts of a geodesic in $Y$   intersecting at $\alpha_q(t_{0})$, contradiction. Therefore $\alpha_q$ is always contained in the regular stratum.

By induction on the stratification, one can now prove that $\vec{\xi}$ restricted to each stratum is foliated, and for every $q\in \newN$ the horizontal geodesic $\alpha_q$ is always contained in the same stratum.
\end{proof}

\begin{corollary}
\label{corollary-Forbitlike-homogeneous}
Let $\F$ be an orbit-like foliation. Let $q$ be a singular point, $Y$ a slice in the stratification $\Sigma$ that contains $q$ and $\newN$
the reduction along $Y$. Then there exists a compact group $G$ acting on $\newN$ such that the leaves of  $\FNo$ are orbits of $G$, i.e., $\FNo$ is  a homogenous SRF on $(\newN,\tilde{\metric}).$ In particular, if $\F$ is closed, then   $B_\e\satsub /\F$ is equal to $(\newN/G)/ \H $
\end{corollary}
\begin{proof}
Let $G=K_q$ be the associated group (recall Definition \ref{definition-locally-homogeneous}). By flowing along the foliated basic vector fields defined in Proposition \ref{lemma1-proposition-vectorfieldY-on-N}, we can make $G$ act smoothly on the whole $\newN$, even though not by isometries.
\end{proof}


\begin{proposition}
\label{lemma-submersion-N}
There exists a metric $\mathrm{g}_\newN$ on $\newN$ with following properties:
\begin{enumerate}
\item[(a)] The submersion $\Projection_Y: \newN\to Y$ is Riemannian.
\item[(b)] Each fiber $\newN_p$ is flat.
\item[(c)] The foliation $\FNo$ is a SRF on $(\newN,\mathrm{g}_\newN).$
\end{enumerate}
\end{proposition}
\begin{proof}
Consider the regular distribution $\Sl$ defined as $\Sl_{z}:=T_{z} S_{p}$ where $z\in \newN_{p}$ and the metric $\mathrm{g}_0$ on $\Sl$ so that 
$\dd (exp_{p})_{v}: T_{v}T_{p}S_{p}\to \Sl_{\exp_{p}(v)}$ is an isometry, for each $v$ normal to $Y$. 
As before, let $\Pi: TM|_{\newN}\to T\newN$ be the orthogonal projection (with respect to original metric) and define a metric on $T\newN$ as $(\Pi|_{\mathcal{H}}^{-1})^{*} \mathrm{g}_{0}.$
Let us denote this new metric on $\newN$ also as $\mathrm{g}_{0}.$
Following \cite[Proposition 2.17]{Alex6} we conclude that $\FNo$ is a SRF on $(\newN,\mathrm{g}_{0})$. Moreover, since $T_{z}\newN_{p}\subset \Sl_{z}$ and $\dd (exp_{p})_{v} T_{v}T_{p}S_{p}\to \Sl_{\exp_{p}(v)}$ is an isometry, every fiber $\newN_p$ is flat with respect to $\mathrm{g}_{0}$.
Consider $H$ the distribution orthogonal to the fibers of $\newN$. We will change the metric of $H$ in order to get the appropriate metric $\mathrm{g}_\newN$ satisfying (a) and (c). Let
\[
\mathrm{g}_\newN:= \mathrm{g}_0|_{H^\perp}+\Projection_Y^* \mathrm{g}_Y
\]

Notice that the metric on the fibers of $\Projection_Y$ is still $\mathrm{g}_0$, thus condition (b) remains satisfied. Moreover, the submersion is now Riemannian by construction.

In order to prove that $\FNo$ is a singular Riemannian foliation, it is enough to compute the Lie derivative $\mathcal{L}_{\vec{X}}\mathrm{g}_{\newN}$ and check that $\mathcal{L}_{\vec{X}}\mathrm{g}_{\newN}|_{H}=0$  for each vector field $\vec{X}$ tangent to the leaves. To this scope, let $\vec{\xi}_1, \vec{\xi}_2$ be $\Projection_Y$-basic vector fields. By Proposition \ref{lemma1-proposition-vectorfieldY-on-N} these vectors are foliated, and in particular for every vector field $\vec{X}$ tangent to the leaves, $[\vec{X},\vec{\xi}_i]$ is tangent to the leaves as well, $i=1,2$. Since clearly $\vec{\xi}_1, \vec{\xi}_2\in H$, we can compute
\begin{eqnarray*}
\mathcal{L}_{\vec{X}}\mathrm{g}_\newN(\vec{\xi}_1,\vec{\xi}_2)&=& \vec{X}\cdot\mathrm{g}_\newN(\vec{\xi}_1,\vec{\xi}_2)- \mathrm{g}_\newN\left([\vec{\xi}_1,\vec{X}],\vec{\xi}_2\right) - \mathrm{g}_\newN\left(\vec{\xi}_1,[\vec{\xi}_2,\vec{X}]\right)\\
&=& \vec{X}\cdot\mathrm{g}\left({\Projection_Y}_*\vec{\xi}_1,{\Projection_Y}_*\vec{\xi}_2\right)-0-0=0.
\end{eqnarray*}
Since $\FNo$ is a SRF with respect to some metric (e.g., the metric $\tilde{g}$ constructed before), by \cite[Proposition 2.14]{Alex6} $\FNo$ is a SRF with respect to $\mathrm{g}_{\newN}$.
\end{proof}


\begin{remark}
The metric $\mathrm{g}_N$ and the metric  $\tilde{\mathrm{g}}$  on $\newN$ are used in Section  \ref{section proof-smoothflow} and  Section \ref{sec-proof-molino},  respectively.
\end{remark}

 Notice that the metric $\mathrm{g}_{\newN}$ does not preserve the transverse metric of $\tilde{\mathrm{g}}$. In particular, an isometry $\phi:(\newN,\tilde{\mathrm{g}})/\FNo\to (\newN,\tilde{\mathrm{g}})/\FNo$ will not be an isometry of $(\newN,\mathrm{g}_{\newN})/\FNo$. Nevertheless, we still have the following result.

\begin{proposition}\label{proposition-property-gN}
Let $\phi:(\newN,\tilde{\mathrm{g}})/\FNo\to (\newN,\tilde{\mathrm{g}})/\FNo$ be an isometry preserving $Y^*$. Then $\phi$ preserves the fibers of $\Projection_{Y^*}$, and
\[
\phi\big|_{\newN_{p_1}/\FNo}: (\newN_{p_1},\mathrm{g}_{\newN})/\FNo\longrightarrow (\newN_{p_2},\mathrm{g}_{\newN})/\FNo
\]
is still an isometry.
\end{proposition}
\begin{proof}
The metric projection $\Projection_{Y^*}$ sends a point $q^*$ to the point $p^*\in Y^*$ which is closest to $q^*$. This is a metric condition, and since $\phi$ preserves the metric, in particular it preserves the fibers of $\Projection_{Y^*}$.

Given $\lambda\in (0,1)$ the homothetic transformation $h_{\lambda}:\newN_p\to \newN_p$, $\exp_pv\longmapsto \exp_p\lambda v$ is a foliated map (cf. \cite{Molino}) and one can define $\tilde{\mathrm{g}}_{\lambda}:=\frac{1}{ \lambda^2}h_{\lambda}^*\tilde{\mathrm{g}}$ such that $(\newN_{p},\tilde{\mathrm{g}}_{\lambda},\FN)$ is still a singular Riemannian foliation. Moreover, since 
\[
\phi\big|_{\newN_{p_1}/\FNo}: (\newN_{p_1},\tilde{\mathrm{g}})/\FNo\longrightarrow (\newN_{p_2},\tilde{\mathrm{g}})/\FNo
\]
is an isometry, it will still be an isometry with respect to $\tilde{\mathrm{g}}_{\lambda}$. Since the restrictions of $\tilde{\mathrm{g}}_{\lambda}$ to the fibers of $\Projection_Y$ converge smoothly to the metric $\mathrm{g}_\newN$, the proposition is proved. See a similar argument in \cite[Theorem 1.2]{AlexRadeschi}.
\end{proof}

\begin{remark}\label{remark-property-gN}
Suppose that the leaves in $\satsub$ meet $Y$ only once, for example in the proof of Theorem \ref{main-theorem}. In this case the isometric action of the singular holonomy pseudogroup (cf. Definition \ref{singular-holonomy-pseudogroup}) $\H$ on $\newN/\FNo$ (as in Section \ref{subsection-disconnected-leaves}) preserves the fibers $\newN_p/\FNo$. Moreover given an isometry $\phi:(\newN,\tilde{\mathrm{g}})/\FN\to (\newN,\tilde{\mathrm{g}})/\FN$, Proposition \ref{proposition-property-gN} can be reproved after replacing $\FNo$ by $\FN$. In particular, $\phi$ induces $\mathrm{g}_{\newN}$-isometries
\begin{equation}
\phi\big|_{\newN_{p_1}/\FN}: (\newN_{p_1},\mathrm{g}_{\newN})/\FN\longrightarrow (\newN_{p_2},\mathrm{g}_{\newN})/\FN.
\end{equation}
whenever $\phi(p_1)=p_2$.
\end{remark}
%





\bibliographystyle{amsplain}

\end{document}